\documentclass{daj}

\dajAUTHORdetails{%
  title = {Automorphisms of shift spaces and the Higman--Thompson groups: the one-sided case}, 
  author = {Collin Bleak, Peter J. Cameron, and Feyishayo Olukoya},
  plaintextauthor = {Collin Bleak, Peter J. Cameron, and Feyishayo Olukoya},
  plaintexttitle = {Automorphisms of shift spaces and the Higman--Thompson groups: the one-sided case}, 
  runningtitle = {Decomposing one-sided shift automorphisms}, 
  runningauthor = {C. Bleak, P. J. Cameron, and F. Olukoya},
  keywords = {54H15, 28D15, 22F50, 68Q99, dynamics, group theory, generating sets, strongly synchronizing automata, counting strongly synchronizing automata, transducers, automorphisms of the one-sided shift, Higman--Thompson groups}
}   

\dajEDITORdetails{%
   year={2021},
   number={15},
   received={8 May 2020},   
   published={20 September 2021},  
   doi={10.19086/da.28243},       
}   

\usepackage{amsmath,amsthm,mathtools,amsfonts,amssymb,verbatim,curves,enumitem, bbm, bm, mathrsfs, float}
\usepackage{tikz}
\usetikzlibrary{arrows,automata,positioning,shapes,shadows}

\newtheorem{theorem}{Theorem}[section]
\newtheorem{prop}[theorem]{Proposition}
\newtheorem{lemma}[theorem]{Lemma}
\newtheorem{cor}[theorem]{Corollary}

\newtheorem{unnumber}{}

\newtheorem{prepf}[unnumber]{Proof\!\!}
\renewenvironment{proof}{\prepf\rm}{\endprepf}

\theoremstyle{definition}
\newtheorem{prerk}[theorem]{Remark}

\theoremstyle{definition}

\theoremstyle{definition}
\newtheorem{ntn}[theorem]{Notation}

\theoremstyle{remark}

\renewcommand{\qed}{\hfill$\Box$}

\newcommand{\End}{\mathop{\mathrm{End}}\nolimits}
\newcommand{\homeo}{\mathop{\mathrm{Homeo}}\nolimits}
\newcommand{\core}{\mathop{\mathrm{core}}\nolimits}
\newcommand{\xn}{X_{n}}
\newcommand{\xnp}{X_{n}^{+}}
\newcommand{\xns}{\xn^{\ast}}
\newcommand{\xnz}{X_n^{\Z}}
\newcommand{\xnN}{X_n^{-\N}}
\newcommand{\xnn}{X_n^{\N}}

\newcommand{\spn}[1]{\widetilde{\mathcal{P}_{#1}}}
\newcommand{\pn}[1]{\mathcal{P}_{#1}}
\newcommand{\shn}[1]{\widetilde{\mathcal{H}_{#1}}}
\newcommand{\hn}[1]{\mathcal{H}_{#1}}

\newcommand{\gen}[1]{\langle #1 \rangle}
\newcommand{\out}[1]{\mbox{Out}({#1})}
\newcommand{\shift}[1]{\sigma_{#1}}
\newcommand{\Shift}[1]{\mathfrak{S}_{#1}}
\newcommand{\rev}[1]{\overleftarrow{#1}}
\newcommand{\im}[1]{\mbox{image}{(#1)}}

\newcommand{\Z}{\mathbb{Z}}
\newcommand{\N}{\mathbb{N}}
\newcommand{\spnprod}[1]{\ast_{\spn{n}}}

\newcommand{\cone}[1]{[#1]}

\newcommand{\CCnr}{\mathfrak{C}_{n,r}}

\DeclareMathOperator{\Sym}{Sym}
\newcommand{\sym}[1]{\Sym(#1)}

\DeclareMathOperator{\Aut}{Aut}
\newcommand{\aut}[1]{\Aut(#1)}

\newcommand{\dual}[1]{{#1}^{\vee}}
\newcommand{\dpi}{\dual{\pi}}

\newcommand{\dlambda}{\dual{\lambda}}

\newcommand{\seteq}{:=}
\begin{document}
\begin{frontmatter}[classification=text]

\title{Automorphisms of shift spaces and the Higman--Thompson groups: the one-sided case} 

\author[collin]{Collin Bleak}
\author[peter]{Peter J. Cameron}
\author[shayo]{Feyisayo Olukoya\thanks{The authors are all grateful for support from EPSRC research grant EP/R032866/1;
the third author also gratefully acknowledges support from Leverhulme Trust Research Project Grant RPG-2017-159}
}

\begin{abstract}
Let $1 \le r < n$ be integers.  We give a proof that the group $\aut{X_{n}^{\N}, \sigma_{n}}$ of automorphisms of the one-sided shift on $n$ letters embeds naturally as a subgroup $\hn{n}$ of the outer automorphism group $\out{G_{n,r}}$ of the Higman--Thompson group $G_{n,r}$. From this, we can represent the elements of $\aut{X_{n}^{\N}, \sigma_{n}}$ by finite state non-initial transducers admitting a very strong synchronizing condition.

 Let $H \in \mathcal{H}_{n}$ and write $|H|$ for the number of states of the minimal transducer representing $H$. We show that $H$ can be written as a product of at most $|H|$ torsion elements.  This result strengthens a similar result of Boyle, Franks and Kitchens, where the decomposition involves more complex torsion elements and also does not support practical \textit{a priori} estimates of the length of the resulting product.

We also explore the number of foldings of de Bruijn graphs and give a counting result for these for word length 2 and alphabet size $n$.

Finally, we offer new proofs of some known results about $\aut{X_{n}^{\N}, \sigma_{n}}$.

\end{abstract}
\end{frontmatter}

\section{Introduction}

Let $1\leq r<n$ be integers.  In this article, we prove that the group $\aut{\xnn, \shift{n}}$ of automorphisms of the one-sided full shift is isomorphic to a subgroup $\hn{n}$ of the group of outer automorphisms of the Higman--Thompson  groups $G_{n,r}$.  Using this embedding we are able to study $\aut{\xnn, \shift{n}}$ from a new perspective.

Fix an alphabet $X_{n}:=\{0,1,2\ldots, n-1\}$ of size $n$. The shift map $\shift{n}$ on the Cantor space of infinite sequences $X_{n}^{\N}$ is the map which shifts a sequence to the left; i.e., a point that was formerly at index $i+1$ now occupies the index $i$. An automorphism of the dynamical system $(\xn^{\N}, \shift{n})$, is a homeomorphism of $\xn^{\N}$ that commutes with the map $\shift{n}$. The collection of all such automorphisms forms a group $\Aut(\xn^{\N}, \shift{n})$. We refer to this group as the group of automorphisms of the shift dynamical system.  

The group $\aut{\xnn, \shift{n}}$ has been well studied (although many questions about it remain). For instance, the seminal paper of Hedlund \cite{Hedlund69} shows that elements of this group can be represented by \emph{sliding block codes} requiring no future information. 
In the same paper, as mentioned above, it is shown that if $n=2$, this group is isomorphic to the cyclic group of order 2; in the paper \cite{BoyleFranksKitchens} the finite subgroups of $\aut{\xn^{\N}, \shift{n}}$ are characterized,  and a full description of the numbers which arise as the order of some torsion element is also given. 

The paper \cite{AutGnr} gives a description of $\out{G_{n,r}}$ as a particular group of non-initial \emph{transducers}.  Note that here a transducer is a finite state machine where each state reads an element from an input alphabet, possibly changes state, and writes a string from an output alphabet. Let $T$ be such a transducer. We call the number of states of $T$ the {size} of $T$ and denote this by $|T|$.

While realizing elements of $\aut{X_n^\N,\sigma_n}$ by transducers has been seen before (see  \cite{BoyleFranksKitchens, GNSenglish}), our realization takes advantage of extra structure arising from a small category of ``folded'' de Bruijn graphs. These are a special set of labeled directed graphs each admitting a synchronizing condition stronger than that appearing in the literature around the Road Colouring Problem and the \u{C}ern\'y Conjecture. We refer to these as \emph{strongly synchronizing automata}, below. 
Using this structure, we give a combinatorial proof of the following theorem (see Theorem \ref{Thm:decomposition} for the more detailed statement). 

 \begin{theorem}\label{Thm:decompositionIntro}
Let $n>1$ be an integer.    An element $T \in \hn{n} \cong \aut{\xnn, \shift{n}}$ can be written as a product of at most $|T|$ elements of $\hn{n}$ arising from automorphisms of directed graphs which are quotients of the underlying graph of $T$.
\end{theorem}

This result is an improvement on a similar result in \cite{BoyleFranksKitchens}. There, in order to decompose an element $T$ of $\aut{\xnn, \shift{n}}$ as a product of torsion elements, one first needs to construct, in the best case, a graph with vertex size of the order of $n^{|T|}$, and it is unclear at the end how many torsion elements one ends up with in the decomposition. Our decomposition on the other hand begins with the transducer $T$ and at each step $i$, produces a torsion factor $H_i$ of $T$ with strictly fewer states than $T$.  

We also  give new combinatorial arguments for the following two results (see Section \ref{sec:FiniteSubgroups}).

\begin{itemize}
	\item   Any finite subgroup of $\aut{\xnn, \shift{n}}$ is isomorphic to a  subgroup of automorphisms of a folded de Bruijn graph. For any such strongly synchronizing automaton the  group of label ignoring automorphisms  embeds as a subgroup of $\aut{\xnz,\shift{n}}$. For the full one-sided shift, the  directed graphs arising from \emph{state splitting} as described in \cite{BoyleFranksKitchens} and \cite{JAshley} are actually  unlabeled directed graphs of strongly synchronizing automata when the directions of the arrows  are reversed. Thus, this embedding result is implicit in \cite{BoyleFranksKitchens,JAshley}.   

	\item When $n=2$,  the unlabeled directed graph corresponding to a strongly synchronizing automaton over a $2$ letter alphabet either has trivial automorphism group or its automorphism group is isomorphic to the cyclic group of order $2$. This gives a new proof  of a classic result of Hedlund \cite{Hedlund69} that $\aut{X_{2}^{\N}, \shift{2}} \cong C_2$. (Note that in \cite{BoyleFranksKitchens} it is shown that when $n>2$ that $\aut{X_{n}^{\N}, \shift{n}}$ contains a non-abelian free group.)
\end{itemize}

Our next result is the promised embedding of $\aut{\xnn, \shift{n}}$ in $\out{G_{n,r}}$ (given in Section \ref{Section:transducers}).  Recall that the Higman--Thompson groups $G_{n,r}$, for $1 \le r < n$, are among the first examples of finitely presented infinite simple groups (when $n$ is even $G_{n,r}$ is simple, and otherwise its derived subgroup is simple, see \cite{Higmanfpsg}).  
\begin{theorem}
	Let $1\leq r<n$ be integers, then $\aut{\xnn, \shift{n}}$ embeds as a subgroup of $\out{G_{n,r}}$.
\end{theorem}

We briefly discuss the strategy of the proof.

 A synchronous transducer that satisfies the strong synchronizing condition induces in a natural way a shift commuting map on $\xnn$. The subgroup of $\out{G_{n,r}}$ consisting of synchronous transducers that induce automorphisms of $(\xnn, \shift{n})$ is what is denoted in the paper \cite{AutGnr} as $\hn{n}$. (A result of \cite{AutGnr} asserts that $\hn{n}$ does not depend on $r$.) The action of $\hn{n}$ on $\xnn$ yields an injective homomorphism  to the group $\aut{\xnn, \shift{n}}$. In order to show that this map is onto, we use the characterization by Hedlund of automorphisms of $(\xnn, \shift{n})$ as sliding block codes which require no past information; we show that a sliding block code with no past information can be simulated by a strongly synchronizing transducer. Thus, we show that $\aut{\xnn, \shift{n}}$ is isomorphic to the group $\hn{n}$ of bi-synchronizing synchronous transducers. It is in the framework of this group $\hn{n}$, that we prove the results stated above.

As mentioned above, in the discussion of the group $\aut{\xn^{\N},\shift{n}}$, there arises an interesting family of small categories of automata and foldings between them.  The automata in any such category are what we call strongly synchronizing automata, below, and are a finite set of natural quotients of some particular de Bruijn graph (\cite{deBruijn}).  The categories are organized in a two-parameter family, and our final result (given in Section \ref{sec:counting}) is to count the number of elements in any such category when one of the parameters is less than or equal to $2$, extending earlier  results from \cite{BleakCameronCount}. The \textit{Bell number} $B(a)$, the number of partitions of a set of size $a$, naturally occurs in the obtained formula. 


\begin{theorem}
	The number of foldings of the de Bruijn graph with word length $2$ over an
	alphabet of cardinality $n$ is
	\[\sum_\pi\prod_{i=1}^{|\pi|}R(|\pi|,|A_i|),\]
	where $\pi$ runs over partitions of the alphabet, $A_i$ is the $i$th part, and
	\[R(s,t)=\sum_\rho(-1)^{|\rho|-1}(|\rho|-1)!\prod_{i=1}^{|\rho|}B(|C_i|s),
	\]
	where $\rho$ runs over all partitions of $\{1,\ldots,t\}$, and $C_i$ is the
	$i$th part.
\end{theorem}

\section{The Curtis, Hedlund, Lyndon Theorem}\label{section:preliminary}
In this paper, unlike the paper of Hedlund~\cite{Hedlund69}, operators will be
on the right of their arguments; but sequences will be indexed from left to
right in the usual way.

We begin with some basic definitions and notation.

We denote by $X_n$ the $n$-element set $\{0,1,\ldots,n-1\}$. Then $X_n^*$
denotes the set of all finite strings (including the empty string
$\varepsilon$) consisting of elements of $X_n$. For an element $w\in X_n^*$,
we let $|w|$ denote the length of $w$ (so that $|\varepsilon|=0$). We further
define
\[X_n^+=X_n^*\setminus\{\varepsilon\},\quad
X_n^k=\{w\in X_n^*:|W|=k\},\quad X_n^{\le k}=\bigcup_{0\le i\le k}X_n^i.\]
We denote the concatenation of strings $x,y\in X_n^*$ by $xy$; in this notation
we do not distinguish between an element of $X_n$ and the corresponding
element of $X_n^1$.

For $x,x_1,x_2\in X_n^*$, if $x$ is the concatenation $x_1x_2$ of $x_1$ and
$x_2$, we write $x_2=x-x_1$.  One can think of the minus operator as ``subtracting off a prefix''.

A bi-infinite sequence is a map $x:\mathbb{Z}\to X_n$. We sometimes write this
sequence as $\dots x_{-1}x_0x_1x_2\dots$, where $x_i=x(i)\in X_n$ (we use left actions for determining sequences). We denote
the set of such sequences by $\xnz$. In a similar way, we define a \emph{ (positive) singly-infinite sequence} as a map   $x: \N \to X_n$ (where, by convention,
$0\in \N$). We write such a sequence as  $x_0x_1x_2\ldots$ and denote the set of all such maps as $\xnn$. Finally, we also set $\xnN$ for the set of all maps $x: \xnN \to \xn$ (the (negative) singly infinite sequences).  Such a map will be written as a sequence $\ldots x_{-2}x_{-1}x_{0}$.

Normally, one thinks of a  full one-sided shift as $(\xnn, \sigma_n)$, where the shift operator $\sigma_n$ operates as $y=x\sigma_n$, where $y_i = x_{i+1}$ for all $i\in \N$. However, in our context it will be much more natural to think of the one-sided shift space as $(\xnN,\shift{n})$, where the shift operator $\shift{n}$ operates as $y= x\shift{n}$, where $y_{i} = x_{{i-1}}$ for all $i\in -\N$.  In Hedlund's characterization, the automorphisms of $(\xnN,\sigma_n)$ are sliding block codes that rely on no future information, instead of no past information.  This will ease many notational difficulties later on.

We can concatenate a string $y\in X_n^*$ with a singly infinite string
$x\in \xnN$, by adding $y$ as a suffix to $x$. We will sometimes subtract a finite string $y$ from a singly infinite string $x$ which has $y$ as a suffix by deleting the suffix $y$.

For a string $\nu \in \xns$ we write $\cone{\nu}$ for the set of all elements of $\xnN$ with $\nu$ as a suffix. Clearly $\cone{\varepsilon} = \xnN$.

Let $F(X_n,m)$ denote the set of functions from $X_n^m$ to $X_n$. Then, for
all $m,r>0$, and all $f\in F(X_n,m)$, we define a map $f_r:X_n^{m+r-1}\to X_n^{r}$
as follows. 
\begin{quote}
Let $x=x_{-m-r+2}\ldots x_{0}$. For $-r+1 \le i\le 0$, set
$y_i=(x_{i-m+1}x_{i-m+2}\ldots x_{i})f$. Then $xf_r=y$, where $y=y_{-r+1}\ldots y_{0}$.
\end{quote}

In other words, we take a ``window''
of length $m$ which slides along the sequence $x$, and at the $i$th step
we apply $f$ to the symbols visible in the window. (One may think of the map as acting on the rightmost letter in the viewing window, with $m-1$ digits of history.) This procedure can
be extended to define a map $f_\infty:\xnz\to\xnz$, by setting $xf_\infty=y$
where $y_i=(x_{i-m+1}\ldots x_{i})f$ for all $i\in\mathbb{Z}$; and similarly for
$\xnN$.

A function $f\in F(X_n,m)$ is called \emph{right permutive} if, for distinct
$x,y\in X_n$ and any fixed block $a\in X_n^{m-1}$, we have $(ax)f\ne(ay)f$.
Alternatively, the map from $X_n$ to itself given by $x\mapsto(ax)f$ is a permutation for all $a\in X_n^{m-1}$.  Analogously, a function $f \in  F(X_n, m)$ is called \emph{left permutive} if  the map from $X_n$ to itself given by $x\mapsto(xa)f$ is a permutation for all $a\in X_n^{m-1}$. 

We note that, if $f$ is not right permutive, then
the induced map $f_\infty$ from $\xnN$ to itself is not injective. The preceding sentence is false if we replace `right' with `left'. For example, take the map $g\in F(X_3,2)$ defined by $ax \mapsto x$ for all $x \in \{0,1,2\}$ and all $a \in \{0,1\}$; $20 \mapsto 1$, $21 \mapsto 0$ and $22 \mapsto 2$. Then $g$ is right permutive but not left permutive and $g_{\infty}$ is a bijection. 
It is not always the case that a right permutive map $f \in F(X_n, m)$ induces a bijective map $f_{\infty}: \xnn\to \xnn$. For example the map $f \in F(X_3, 2)$ defined by $a0 \mapsto 0$, $a1 \mapsto 2$, $a2 \mapsto 1$ for all $a \in \{0,1\}$; $20 \mapsto 1$, $21 \mapsto 0$, $22 \mapsto 2$ is a right permutive map such that $(\ldots 111 \ldots) f_{\infty} = (\ldots 222 \ldots) f_{\infty}$. We note that a right permutive map always induces a surjective map from $\xnN$ to itself. 

\begin{prerk}
Observe that, if $f\in F(X_n,m)$ and $k\ge1$, then the map $g\in F(X_n,m+k)$
given by $$(x_{-m-k+1}\ldots x_0)g=(x_{-m+1}\ldots x_{0})f,$$ satisfies
$g_\infty=f_\infty$.
\label{F(X_n,m)containedinF(X_n,m+1)}
\end{prerk}

\medskip

The sets $\xnz$, $\xnn$ and $\xnN$ are topological spaces, equipped with the Tychonoff
product topology derived from the discrete topology on $X_n$. Each is
homeomorphic to Cantor space. The set $\{ \cone{\nu} \mid \nu \in \xns\}$ is a basis of clopen sets for the topology on $\xnN$.

In this paper the \emph{shift map} $\shift{n}$ is the map which sends a sequence $x$ in
$\xnz$ or $\xnn$ to the sequence $y$ given by $y(i)=x(i-1)$ for all
$i$ in $\mathbb{Z}$ or $\mathbb{-N}$ respectively.

\medskip

The following result is due to Curtis, Hedlund and Lyndon~\cite[Theorem 3.1]{Hedlund69}: 

\begin{theorem}
Let $f\in F(X_n,m)$. Then $f_\infty$ is continuous on $\xnN$ and $\xnz$ and commutes with the
shift map on $\xnz$ and $\xnN$.
\label{t:hed1}
\end{theorem}

A continuous function from $\xnz$ to itself which commutes with the shift map
is called an \emph{endomorphism} of the shift dynamical system
$(\xnz,\shift{n})$. If the function is invertible, since $\xnz$ is compact and Hausdorff, its inverse is continuous:
it is an \emph{automorphism} of the shift system. The sets of endomorphisms and
of automorphisms are denoted by $\End(\xnz,\shift{n})$ and $\Aut(\xnz,\shift{n})$
respectively. Under composition, the first is a monoid, and the second a group.

Analogously, a continuous
function from $\xnN$ to itself which commutes with the shift map on this
space is an \emph{endomorphism of the one-sided shift} $(\xnN,\shift{n})$; if it is invertible, it is an \emph{automorphism} of this shift system. The sets of
such maps are denoted by $\End(\xnN,\shift{n})$ and $\Aut(\xnN,\shift{n})$; again
the first is a monoid and the second a group.

Note that $\shift{n}\in\Aut(\xnz,\shift{n})$, whereas
$\shift{n}\in\End(\xnN,\shift{n})\setminus\Aut(\xnN,\shift{n})$. More generally, the inclusions $$\End(\xnN, \shift{n}) \subsetneq \End(\xnz, \shift{n}) \mathrm{ \ and \ } \aut{\xnN, \shift{n}} \subsetneq\aut{\xnz, \shift{n}}$$ are valid.

Define
\begin{eqnarray*}
F_\infty(X_n) &\seteq& \bigcup_{m\ge0}\{f_\infty:f\in F(X_n,m)\},\\
RF_\infty(X_n) &\seteq& \bigcup_{m\ge0}\{f_\infty:f\in F(X_n,m), f
\mbox{ is right permutive}\}.
\end{eqnarray*}

Theorem~\ref{t:hed1} shows that $F_\infty(\xn)\subseteq\End(\xnz,\sigma_n)$. In fact $F_{\infty}(\xn)$ and $RF_{\infty}(\xn)$ are submonoids of $\End(\xnz,\sigma_n)$. For given natural numbers $l$ and $m$, $f \in F(X_{n},l)$ and $g \in F(X_{n}, m)$, the function $h \in F(X_{n}, l+m-1)$ defined by $(a_{-l-m+2}\ldots a_{-1}a_{0})h = ((a_{-l-m+2}\ldots a_{-1}a_{0})f_{l+m-1})g$ satisfies $h_{\infty} = f_{\infty} \circ g_{\infty}$. If $f$ and $g$ are both right permutive, then so also is $h$.
Note that $\shift{n}\in F_\infty(X_n)$ since the function $f\in X_n^2$ defined by
\[(x_{-1}x_0)f=x_{-1},\]
satisfies $f_\infty=\shift{n}$. However $\shift{n}^{-1}$ is not an element of $F_{\infty}(X_{n})$. Now, \cite[Theorem 3.4]{Hedlund69} shows:

\begin{theorem}
$\End(\xnz,\shift{n}) = \{ \shift{n}^{i} \phi \mid i \in \mathbb{Z}, \phi \in F_{\infty}(X_{n}) \}$.
\label{t:hed2}
\end{theorem}

The following result is a corollary:
\begin{theorem}
$RF_\infty(\xn)$ is a submonoid of $\End(\xnz, \shift{n})$ and $\Aut(\xnN,\shift{n})$ is the largest inverse closed subset of $RF_\infty(X_n)$.
\end{theorem}

\section{Connections to transducers}\label{Section:transducers}
In this section we will give a definition and very brief history of the Higman--Thompson groups $G_{n,r}$, and an incomplete list of references to related research on these in the literature.  We then flesh out the connection between $\out{G_{n,r}}$ and $\Aut(\xnN,\shift{n})$, explaining that one can represent all elements of these groups by a certain class of finite strongly synchronizing transducers.  It follows from these characterizations that $\Aut(\xnN,\shift{n})$ embeds in a straightforward fashion in $\out{G_{n,r}}$.
\subsection{The Higman--Thompson groups $G_{n,r}$}
Richard Thompson's notes from 1965 \cite{ThompsonNotes} introduce three infinite finitely presented groups, now commonly known as $F\leq T\leq V$, and show that $T$ and $V$ are also simple. These are the first known examples of infinite finitely presented simple groups.  See \cite{CFP} for a standard survey on the Thompson groups.

The Higman--Thompson groups $G_{n,r}$ were introduced by Higman in \cite{Higmanfpsg}, where he generalizes Thompson's construction of the group $V$ from \cite{ThompsonNotes} ($V\cong G_{2,1}$).  Calculations in the Higman--Thompson groups do not play a role in the main body of this article, but for the curious reader, in Subsubsection \ref{sssec:gnr} we still provide an oft-used concrete realization of the groups $G_{n,r}$ as specific groups of homeomorphisms of Cantor spaces.

  In any case, Higman in \cite{Higmanfpsg} shows that the groups $G_{n,r}$ are simple when $n$ is even, and that they have a simple commutator subgroup $G_{n,r}'$ of index two when $n$ is odd.  Thus he obtains the first known infinite family of infinite, finitely presented simple groups (Higman also shows there are infinitely many isomorphism types amongst the groups $G_{n,r}$).  See \cite{Higmanfpsg,BrownGeometryFPSimple,BirgetCircuits,Pardo,Martinez-PerezMatucciNucinkis,DicksMartinezPerez,AutGnr,SzymikWahlHomologyHigThomp} for some research on these groups.  The paper \cite{AutGnr} provides the characterization of Out($G_{n,r}$) that is relevant to this article.
  
  It is expressed in \cite{BleakQuick} that R. Thompson's group $V$ represents a ``gentle'' realization of the alternating groups $A_k$ into an infinite context.  In Subsubsection \ref{sssec:gnr}  we generate a given group $G_{n,r}$ by transpositions of clopen subsets of Cantor space (which themselves can each be written as a product of $n$ ``smaller'' transpositions).  This type of generation is part of the justification for this view on the fundamental nature of the groups $G_{n,r}$ (or at least, for the group $V\cong G_{2,1}$).
\subsubsection{The Cantor spaces $\CCnr$}
Let $r$ and $n$ be given natural numbers with $n\geq 2$.
\newcommand{\nset}{\{0,1,\ldots,n-1\}}
\newcommand{\rd}{\dot{r}}
\newcommand{\CCn}{\mathfrak{C}_{n}}

Given such $r$ and $n$, determine two finite alphabets $\rd\seteq\{\dot{0},\dot{1},\ldots, \dot{r-1}\}$ and our standard $n$-letter alphabet $\nset$ and give each of these sets the discrete topology.  A standard characterization of $n$-ary Cantor space $\mathfrak{C}_n$ is as the space $\mathfrak{C}_n\seteq\nset^{\N}$ under the product topology, and we can then form another Cantor space $\CCnr$ as the product $\rd\times\CCn$.
\[
\CCnr:=\{c a_0 a_1 a_2\ldots \mid a_i \in \nset, c\in \rd\}. 
\]
That is, $\CCnr$ can be thought of as a disjoint union of $r$ copies of the standard infinite $n$-ary Cantor space $\CCn$.

The space $\CCnr$ has the topology generated by using \emph{cones} as basic open sets.  A cone is any subset of the form $\mathbf{x}\CCn$, where
$\mathbf{x}$ is a finite sequence of the form
$cx_0x_1\ldots x_j$ for $c\in \rd$ and each $x_i\in \nset$.  Thus, the cone at $\mathbf{x}$ is all of 
the elements of $\CCnr$ which have leading prefix equal to $\mathbf{x}$.

\newcommand{\Gnr}{G_{n,r}}
\subsubsection{The Higman--Thompson groups $G_{n,r}$}\label{sssec:gnr}
The group $\Gnr$ can be realized as a subgroup of $\homeo(\CCnr)$ as follows. It is the subgroup generated by \emph{prefix replacement swaps}: one specifies two incomparable finite prefixes (that is, neither is a prefix of the other), say ${\bf x}\seteq c_1a_0a_1\ldots a_j$ and ${\bf y}\seteq c_2b_0b_1\ldots b_k$ (for some $c_1,c_2\in \rd$ and with each $a_i$ and $b_i$ from the alphabet $\{0,1,2,\ldots,n-1\}$), and then interchanges the cones in the Cantor space $\CCnr$ determined by these prefixes, using an element we denote by $(\mathbf{x}\;\;\mathbf{y})$.  For example, for this particular swap we have:
\begin{align*}
c_1a_0a_1\ldots a_ja_{j+1}a_{j+2}\ldots \;(\mathbf{x}\;\;\mathbf{y})&=c_2 b_0b_1\ldots b_ka_{j+1}a_{j+2}\ldots\\
c_2b_0b_1\ldots b_kb_{k+1}b_{k+2}\ldots \;(\mathbf{x}\;\;\mathbf{y})&= c_1a_0a_1\ldots a_jb_{k+1}b_{k+2}\ldots\\
\vec{z} \;(\mathbf{x}\;\;\mathbf{y})&= \vec{z} \textrm{ where }\vec{z}\textrm{ does not have prefix }\mathbf{x} \textrm { or }\mathbf{y}.
\end{align*} 

Composition of these swaps results in homeomorphisms of $\CCnr$  that replace some finite decomposition $D$ of $\CCnr$ into pairwise disjoint cones by some other such decomposition $R$, given a bijection from the set $D$ of cones to the set $R$ of cones.  This works as follows: one replaces the maximal common prefix of the points in a cone appearing in $D$ by the maximal common prefix of the points in the corresponding cone from $R$.

\subsection{Automata and transducers} \label{subsection:transducers}

An \emph{automaton}, in our context, is a triple $A=(X_A,Q_A,\pi_A)$, where
\begin{enumerate}
\item $X_A$ is a finite set called the \emph{alphabet} of $A$ (we assume that
this has cardinality $n$, and identify it with $X_n$, for some $n$);
\item $Q_A$ is a finite set called the \emph{set of states} of $A$;
\item $\pi_A$ is a function $X_A\times Q_A\to Q_A$, called the \emph{transition
function}.
\end{enumerate}

The \emph{size} of an automaton $A$ is the cardinality of its state set.  We use the notation $|A|$ for the size of the $A$.

We regard an automaton $A$ as operating as follows. If it is in state $q$ and
reads symbol $a$ (which we suppose to be written on an input tape), it moves
into state $\pi_A(a,q)$ before reading the next symbol. As this suggests, we
can imagine that the automaton $A$ is in the middle of an input word, reads the next letter and  moves to the right, possibly changing state in the process.

We can extend the notation as follows. For $w\in X_n^m$, let $\pi_A(w,q)$ be
the final state of the automaton which reads the word $w$ from initial state $q$. Thus, if $w=x_0x_1\ldots x_{m-1}$, then
\[\pi_A(w,q)=\pi_A(x_{m-1},\pi_A(x_{m-2},\ldots,\pi_A(x_0,q)\ldots)).\]
By convention, we take $\pi_A(\varepsilon,q)=q$.

For a given state $q\in Q_A$, we call the automaton $A$ which starts in
state $q$ an \emph{initial automaton}, denoted by $A_q$, and say that it is
\emph{initialized} at $q$.

An automaton $A$ can be represented by a labeled directed graph, whose
vertex set is $Q_A$; there is a directed edge labeled by $a\in X_A$ from
$q$ to $r$ if $\pi_A(a,q)=r$.

A \emph{transducer} is a quadruple $T=(X_T,Q_T,\pi_T,\lambda_T)$, where
\begin{enumerate}
\item $(X_T,Q_T,\pi_T)$ is an automaton;
\item $\lambda_T:X_T\times Q_T\to X_T^*$ is the \emph{output function}.
\end{enumerate}
Such a transducer is an automaton which can write as well as read; after
reading symbol $a$ in state $q$, it writes the string $\lambda_T(a,q)$ on an
output tape, and makes a transition into state $\pi_T(a,q)$. We call the automaton $(X_{T}, Q_{T}, \pi_{T})$ the underlying automaton of $T$. Thus, the size of a transducer is the size of its underlying automaton. An \emph{initial
transducer} $T_q$ is simply a transducer which starts in state $q$.  Transducers which are \emph{synchronous} (i.e., which always write one letter whenever they read one letter) are also known as \emph{Mealy machines} (see \cite{GNSenglish}), although we generally will not use that language here.  Transducers which are not synchronous are described as \emph{asynchronous} when this aspect of the transducer is being highlighted.  In this paper, we will only work with synchronous transducers without an initial state, and, below, \textbf{we will simply call these transducers}.

In the same manner as for automata, we can extend the notation to allow transducers to act on finite strings: we let $\pi_T(w,q)$ and $\lambda_T(w,q)$ be, respectively, the final state and the concatenation of all the outputs obtained when a transducer $T$ reads a string $w$ from a state $q$.

A transducer $T$ can also be represented as an edge-labeled directed graph.
Again the vertex set is $Q_T$; now, if $\pi_T(a,q)=r$, we put an edge with
label $a|\lambda_T(a,q)$ from $q$ to $r$. In other words, the edge label describes both the input and the output associated with that edge. We call $a$ the \textit{input label} of the edge and $\lambda_{T}(a,q)$ the \textit{output label} of the edge.

For example, Figure~\ref{fig:shift2} describes a synchronous transducer over the alphabet
$X_2$.

\begin{figure}[htbp]
\begin{center}
 \begin{tikzpicture}[shorten >=0.5pt,node distance=3cm,on grid,auto]
 \tikzstyle{every state}=[fill=none,draw=black,text=black]
    \node[state] (q_0)   {$a_1$};
    \node[state] (q_1) [right=of q_0] {$a_2$};
     \path[->]
     (q_0) edge [loop left] node [swap] {$0|0$} ()
           edge [bend left]  node  {$1|0$} (q_1)
     (q_1) edge [loop right]  node [swap]  {$1|1$} ()
           edge [bend left]  node {$0|1$} (q_0);
 \end{tikzpicture}
 \end{center}
 \caption{A transducer over $X_2$ \label{fig:shift2}}
\end{figure}
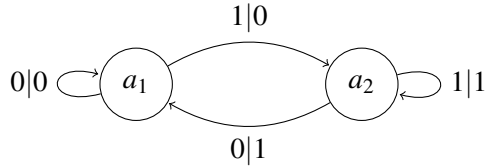



We can regard an automaton, or a transducer, as acting on an infinite string from $\xnn$ where  $X_n$ is the alphabet. This action is given by iterating
the action on a single symbol; so the output string is given by
\[\lambda_T(xw,q) = \lambda_T(x,q)\lambda_T(w,\pi_T(x,q)).\]

Thus $T_q$ induces a map $w\mapsto\lambda_T(w,q)$ from $\xnn$ to itself; it is
easy to see that this map is continuous. If it is a
homeomorphism, then we call the state $q$ a \emph{homeomorphism state}. We write $\im{q}$ for the image of the map induced by $T_{q}$.

Two states $q_1$ and $q_2$ are said to be \emph{$\omega$-equivalent} if the
transducers $T_{q_1}$ and $T_{q_2}$ induce the same continuous map. (This can
be checked in finite time, see~\cite{GNSenglish}.)  More generally, we say that two
initial transducers $T_q$ and $T'_{q'}$ are \emph{$\omega$-equivalent} if they
induce the same continuous map on $\xnn$. 

A transducer is said to be \emph{weakly minimal} if no two states are
$\omega$-equivalent. For a synchronous transducer $T$, two states $q_1$ and $q_2$ are $\omega$-equivalent if $\lambda_T(a, q_1) = \lambda_T(a,q_2)$ for any finite word $a \in X_n^{*}$. Moreover, if $q_1$ and $q_2$ are $\omega$-equivalent states of a synchronous transducer, then for any finite word $a \in X_{n}^+$, $\pi_{T}(a, q_1)$ and $\pi_{T}(a, q_2)$ are also $\omega$-equivalent states. 

There is a stronger notion of minimality which appears in \cite{GNSenglish} and applies also to asynchronous transducer, hence our use of the adjective \emph{weakly}.

 Two weakly minimal non-initial transducers $T$ and $U$  are said to be  \emph{$\omega$-equal} if there is a bijection $f: Q_{T} \to Q_{U}$, such that for any $q \in Q_{T}$, $T_{q}$ is $\omega$-equivalent to $U_{(q)f}$. Two weakly minimal initial transducers $T_{p}$ and $U_{q}$ are said to be $\omega$-equal if there is a bijection  $f: Q_{T} \to Q_{U}$, such that $(p)f = q$ and for any $t \in Q_{T}$, $T_{t}$ is $\omega$-equivalent to $U_{(t)f}$. We shall use the symbol `$=$' to represent $\omega$-equality of initial and non-initial transducers. Two non-initial transducers are said to be $\omega$-equivalent if they have $\omega$-equal minimal representatives.

In the class of synchronous transducers, the  $\omega$-equivalence class of  any transducer has a unique weakly minimal representative. In the general case, if one permits infinite outputs from finite inputs,
Grigorchuk \textit{et al.}~\cite{GNSenglish} prove that the $\omega$-equivalence
class of an initialized transducer $T_q$ has a unique minimal representative and give an algorithm for computing this representative. 

Throughout this article, as a matter of convenience, we shall not distinguish between $\omega$-equivalent transducers. Thus, for example, we introduce various groups as if the elements of those groups are transducers, whereas the elements of these groups are in fact  $\omega$-equivalence classes of transducers. 

Given two transducers $T=(X_n,Q_T,\pi_T,\lambda_T)$ and
$U=(X_n,Q_U,\pi_U,\lambda_U)$ with the same alphabet $X_n$, we define their
product $T*U$. The intuition is that the output for $T$ will become the input
for $U$. Thus we take the alphabet of $T*U$ to be $X_n$, the set of states
to be $Q_{T*U}=Q_T\times Q_U$, and define the transition and rewrite functions
by the rules
\begin{eqnarray*}
\pi_{T*U}(x,(p,q)) &=& (\pi_T(x,p),\pi_U(\lambda_T(x,p),q)),\\
\lambda_{T*U}(x,(p,q)) &=& \lambda_U(\lambda_T(x,p),q),
\end{eqnarray*}
for $x\in X_n$, $p\in Q_T$ and $q\in Q_U$. Here we use the earlier 
convention about extending $\lambda$ and $\pi$ to the case when the transducer
reads a finite string.  If $T$ and $U$ are initial with initial states $q$ and $p$ respectively then the state $(q,p)$ is considered the initial state of the product transducer $T*U$.

In automata theory a synchronous (not necessarily initial) transducer $T = (X_n, Q_{T}, \pi_{T}, \lambda_T)$ is \emph{invertible} if for any state $q$ of $T$, the map $\rho_q:=\lambda_{T}(\centerdot, q): X_{n} \to X_{n}$ is a bijection. In this case the inverse of $T$ is the transducer $T^{-1}$ with state set $Q_{T^{-1}}:= \{ q^{-1} \mid q \in Q_{T}\}$, transition function $\pi_{T^{-1}}: X_{n} \times Q_{T^{-1}} \to Q_{T^{-1}}$ defined by $(x,p^{-1}) \mapsto q^{-1}$ if and only if $\pi_{T}((x)\rho_{p}^{-1}, p) =q$, and output function  $\lambda_{T^{-1}}: X_{n} \times Q_{T^{-1}} \to X_{n}$ defined by  $(x,p) \mapsto (x)\rho_{p}^{-1}$. 

One can the interpret the previous paragraph as follows:  For the invertible synchronous transducer $T$, the inverse transducer $T^{-1}$ is the result of switching inputs and outputs on all transitions of $T$.  In particular, we can think of a synchronous transducer as an ordered pair of automata, each with the same structure as directed graphs.  Inversion then corresponds to swapping the ordering on this ordered pair, much as we do in constructing inverses for non-zero fractions by switching  the numerator and denominator in a non-zero fraction of integers. In the transducer $T$ depicted in Figure \ref{fig-inversionExamp} below, the input automaton corresponds to the directed graph with the input labels on the edges and the output automaton corresponds to the directed graph with the output labels on the edges.   Henceforth, we will refer to the input automaton as the \emph{domain automaton} and the output automaton as the \emph{range automaton}.

\begin{figure}[H]
	\begin{center}
		\begin{tikzpicture}[shorten >= .5pt,node distance=2cm,on grid,auto] 
		\node[state] (q0)   {$q_2$};
		\node[state, xshift= -1.5cm, yshift=-2cm] (q1)   {$q_1$};
		\node[state, xshift=1.5cm, yshift=-2cm ] (q2) {$q_0$};
		\node[state, xshift=7cm] (p0) {$q_2^{-1}$};
		\node[state, xshift= 5.5cm, yshift=-2cm] (p1)   {$q_1^{-1}$};
		\node[state, xshift=8.5cm, yshift=-2cm ] (p2) {$q_0^{-1}$};
		\node[xshift=0cm, yshift=-3.5cm]{$T$};
		\node[xshift=7cm, yshift=-3.5cm]{$T^{-1}$};
		\path[->] 
		
				(q0) edge[out=180, in=90] node[swap]{$1|0$} (q1)
		edge[out=75, in=105, loop] node[swap] {$0|2$} () 
		edge[out=360, in=90] node{$2|1$} (q2)
		
		(q1) edge[in=190, out=80]  node[swap] {$0|2$} (q0)
		edge[out=140, in=170,loop] node[swap] {$1|1$} ()
		edge[in=190, out=350] node[swap] {$2|0$} (q2) 
		
		(q2) edge[in=360, out=180] node[swap] {$0|1$} (q1)
		edge[out=40, in=10,loop] node {$2|0$} ()
		edge[in=350, out=100] node {$1|2$}   (q0)

		(p0) edge[out=180, in=90] node[swap]{$0|1$} (p1)
		edge[out=75, in=105, loop] node[swap] {$2|0$} () 
		edge[out=360, in=90] node{$1|2$} (p2)
		
		(p1) edge[in=190, out=80]  node[swap] {$2|0$} (p0)
		edge[out=140, in=170,loop] node[swap] {$1|1$} ()
		edge[in=190, out=350] node[swap] {$0|2$} (p2) 
		
		(p2) edge[in=360, out=180] node[swap] {$1|0$} (p1)
		edge[out=40, in=10,loop] node {$0|2$} ()
		edge[in=350, out=100] node {$2|1$}   (p0);
		\end{tikzpicture}
	\end{center}
	\caption{Inverting a synchronous transducer $T$.\label{fig-inversionExamp}}
\end{figure}
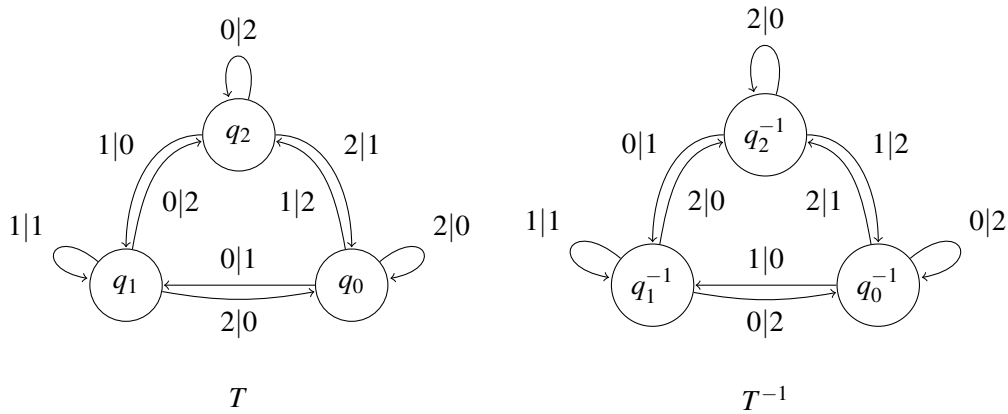

In this article, we will come across synchronous transducers which are not invertible in the automata theoretic sense but which nevertheless induce self-homeomorphisms of the spaces $\xnz$ and  $\xnN$. Consequently it will important to distinguish between an automaton theoretic inverse and the inverse of the induced action on the various spaces we consider.

\subsection{Synchronizing automata and bisynchronizing transducers}

Given a natural number $k$, we say that an automaton $A$ with alphabet $X_n$
is \emph{synchronizing at level $k$} if there is a map
$\mathfrak{s}_{k}:X_n^k\mapsto Q_A$ such that, for all $q$ and any word
$w\in X_n^k$, we have $\pi_A(w,q)=\mathfrak{s}_{k}(w)$. In other words, $A$ is synchronizing at level $k$ if, after reading a word $w$ of length $k$ from a state $q$, the final state depends only on
$w$ and not on $q$. (Again we use the extension of $\pi_A$ to
allow the reading of an input string rather than a single symbol.) We
call $\mathfrak{s}_{k}(w)$ the state of $A$ \emph{forced} by $w$; the map $\mathfrak{s}_{k}$ is called the \emph{synchronizing map at level $k$}. An automaton $A$ is called \emph{strongly synchronizing} if it is synchronizing at level $k$ for some $k$.

We remark here that the notion of synchronization occurs in automata theory
in considerations around the \emph{\v{C}ern\'y conjecture}, in a weaker sense.
A word $w$ is said to be a \emph{reset word} for $A$ if $\pi_A(w,q)$ is
independent of $q$; an automaton is called \emph{synchronizing} if it has
a reset word~\cite{Volkov2008,ACS}. Our definition of ``synchonizing at level $k$''/"strongly synchronizing"
requires every word of length $k$ to be a reset word for the automaton.

If the automaton $A$ is synchronizing at level $k$, we define the
\emph{core} of $A$ to be the set of states forming the image of the map
$\mathfrak{s}$. It is an easy observation that, if $A$ is synchronizing at
level $k$, then its core is an automaton in its own right, and is also 
synchronizing at level $k$. We denote this automaton by $\core(A)$. We say that an
automaton or transducer is \emph{core} if it is equal to its core. Moreover,
if $T$ is a transducer which (regarded as an
automaton) is synchronizing at level $k$, then the core of $T$ (similarly denoted 
 $\core(T)$) induces a continuous map $f_T:\xnz\to\xnz$.

Clearly, if $A$ is synchronizing at level $k$, then it is synchronizing at
level~$l$ for all $l\ge k$; but the map $f_T$ is independent of the level
chosen to define it.

Let $T_q$ be an initial transducer which is invertible with inverse $T_q^{-1}$. If $T_q$ is synchronizing at level $k$, and $T_q^{-1}$ is synchronizing at level $l$, 
we say that $T_q$ is \emph{bisynchronizing} at level $(k,l)$. If $T_q$ is
invertible and is synchronizing at level~$k$ but not bisynchronizing, we say
that it is \emph{one-way synchronizing} at level~$k$.

For a non-initial invertible transducer  $T$ we also say $T$ is bi-synchronizing (at level $(k,l)$) if both $T$ and its inverse $T^{-1}$ are synchronizing at levels $k$ and $l$ respectively.

\begin{ntn}
	Let $T$ be a transducer which is synchronizing at level $k$ and let $l \ge k$ be any natural number. Then for any word $w \in X_{n}^{l}$, we write $q_{w}$ for the state $\mathfrak{s}_{l}(w)$, where $\mathfrak{s}_{l}: X_{n}^{l} \to Q_{T}$ is the synchronizing map at level $l$.
\end{ntn}

The following result was proved in Bleak \textit{et al.}~\cite{AutGnr} .

\begin{prop}
Let $T$, $U$ be  transducers which (as automata) are synchronizing at levels $j$, $k$ respectively, Then $T*U$ is synchronizing at level at most $j+k$.
\label{p:synchlengthsadd}
\end{prop}


In what follows we give a formula specifying how strongly synchronizing transducers act by continuous functions on $\xnz$. The formula induces a natural action on $\xnN$ which immediately commutes with the shift. We recall that in our context the shift map $\shift{n}$ is the map which sends a sequence $x \in \xnN \sqcup \xnz$ to the sequence $y \in \xnN \sqcup \xnz$ given by $y_{i} =  x_{i-1}$ for all valid $i \in -\N \sqcup  \Z$. This represents a deviation from the way the shift map conventionally operates, however, in this point of view, as we will become clear, synchronizing transducers can locally process inputs in a manner consistent with the definition given in Subsection~\ref{subsection:transducers}.  The formula is as follows:

 Let $T$ be a transducer which is core, and is synchronizing at level $k$. The map $f_{T}: \xnz \to \xnz$ maps an element $x \in \xnz$ to the sequence $y$ defined by $y_{i} = \lambda_T(x_i, q_{x_{i-k}x_{i-k+1}\ldots x_{i-1}})$. We  also write $f_{T}$ for the continuous map from $\xnN$ to itself defined by  $y_{i} = \lambda_T(x_i, q_{x_{i-k}x_{i-k+1}\ldots x_{i-1}})$ for all $i \in  -\N$. We note that the induced map on $\xnN$ is simply the restriction of the map on $\xnz$ to the subsequence indexed by the negative integers. 

We note that given an element $x \in \xnz$  such that $(x)f_{T} = y$, then $y_0 y_1 \ldots =  (x_0x_1 \ldots) T_{q_{_{x_{-k}x_{-k+1}\ldots x_{-1}}}}$. This is what was meant by the transducer $T$ acts locally in a manner consistent with the definitions of Subsection~\ref{subsection:transducers}.

Now strongly synchronizing transducers may induce endomorphisms of the shift:

\begin{prop}\label{prop:pntildeisinendo}
Let $T$ be a minimal transducer which is  synchronizing at 
level $k$ and which is core. Then $f_T\in\End(\xnz,\shift{n})$ and $f_T \in \End(\xnN, \shift{n})$.
\end{prop}

\begin{proof}
It is clear from the assumptions that $f_T$ is continuous and by definiton induces a map from $\xnz$ to itself and from $\xnN$ to itself. Now let 
$x\in\xnz \sqcup \xnN$ and $i\in\mathbb{Z}$ an appropriate index for $x$. Let $y=(x)f_T$. Observe that
$y_i=\lambda(x_i,q)$, where $q=\mathfrak{s}({x_{i - k}\ldots x_{i-1}})$
is the state forced by ${x_{i-k}\ldots x_{i-1}}$.

Now let $u=(x)\shift{n}$ and $v=(u)f_T$. Then
\[v_{i-1}=\lambda(u_{i-1},q'),\]
where $q'$ is the state of $T$ forced by ${u_{i-k-1}\ldots u_{i-2}}$. But
by assumption, $u_{i-k-1}\ldots u_{i-2}=x_{i-k}\ldots x_{i-1}$, and this
string forces state $q$; so $q'=q$, and hence $v_{i-1}=y_i$.

It now follows that $(x)f_T\shift{n}=(y)\shift{n}=v=(u)f_T=(x)\shift{n} f_T$.\qed
\end{proof}

The transducer in Figure~\ref{fig:shift2} induces the shift map on $\xnz$. More generally, let  $\Shift{n} = (\xn, Q_{\Shift{n}}, \pi_{\Shift{n}}, \lambda_{\Shift{n}})$ be the transducer defined as follows. Let $Q_{\Shift{n}}:= \{0,1,2,\ldots, n-1\}$, and let $\pi_{\Shift{n}}: \xn \times Q_{\Shift{n}}\to Q_{\Shift{n}}$  and $\lambda_{\Shift{n}}: \xn \times Q_{\Shift{n}}\to \xn$ be defined by $\pi_{\Shift{n}}(x, i) = x$ and $\lambda_{\Shift{n}}(x, i) = i$ for all $x \in X_{n}$ $i \in Q_{\shift{n}}$. Then $f_{\Shift{n}} = \shift{n}$.

In \cite{AutGnr}, the authors show that the set $\spn{n}$ of weakly
minimal finite synchronous core   transducers is a monoid. (Note that core transducers are strongly synchronizing.) The monoid operation
consists of taking the product of transducers and reducing it by removing non-core states and identifying $\omega$-equivalent states to obtain a weakly minimal and synchronous representative. Let $\mathcal{P}_n$ be the subset
of $\spn{n}$ consisting of transducers $T$ such that $f_{T}$ is an automorphism
of the two-sided  shift dynamical system. We observe that elements of $\pn{n}$ may not be minimal. It is clear that $\Shift{n} \in \pn{n}$. For an element $T \in \spn{n}$, we use the language $T$ induces an automorphism of the two-sided shift dynamical system to mean that $f_{T}$ is an element of $\aut{\xnz, \shift{n}}$.

\subsection{De Bruijn graphs and $\End(\xnz, \shift{n})$}

The \emph{de Bruijn graph} $G(n,m)$ can be defined as follows, for integers
$m\ge1$ and $n\ge2$. The vertex set is $X_n^m$, where $X_n$ is the alphabet
$\{0,\ldots,n-1\}$ of cardinality $n$. There is a directed arc from
$a_1\ldots a_{m}$ to $a_2\ldots a_{m}a_0$, with label $a_0$.

Note that, in the literature, the directed edge is from $a_0a_1\ldots a_{m-1}$ to $a_1\ldots a_{m-1}a_m$ and the label on this edge is often given as the
$(m+1)$-tuple $a_0a_1\ldots a_{m-1}a_m$. However, to fit with the notation
already defined, the equivalent definition given above is more apt. 

Figure~\ref{fig-DB-3-2-straight} shows the de Bruijn graph $G(3,2)$.

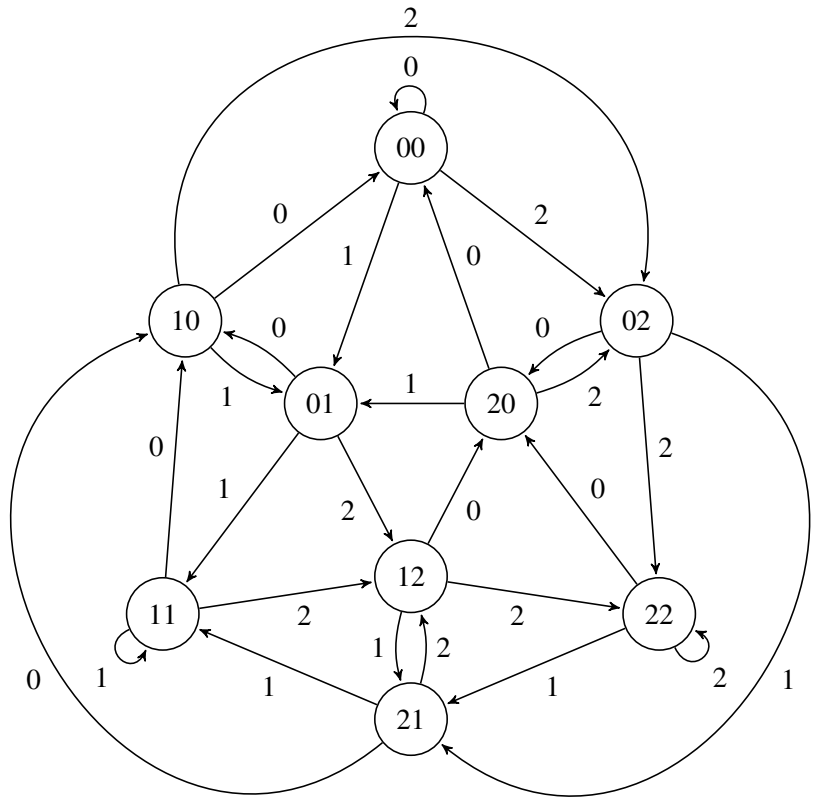
\begin{figure}[htbp]
\begin{center}
\begin{tikzpicture}[->,>=stealth',shorten >=1pt,auto,node distance=2.3cm,on grid,semithick,every state/.style={draw=black,text=black}]
   \node[at={(0,2.9)},state] (a) {$00$}; 
   \node[at={(-3.3,-3.3)},state] (b)  {$11$}; 
   \node[at={(3.3,-3.3)},state] (c) {$22$}; 
   \node[at={(-1.2,-0.5)},state] (d)   {$01$}; 
   \node[at={(-3.0,0.6)},state] (f)  {$10$}; 
   \node[at={(3.0,0.6)},state] (e) {$02$}; 
   \node[at={(1.2,-0.5)},state] (h)  {$20$}; 
   \node[at={(0,-2.8)},state] (g)  {$12$}; 
   \node[at={(0,-4.7)},state] (i){$21$}; 

    \path (a) edge [out=70,in=110,loop,min distance=0.5cm]node [swap]{$0$} (a);
    \path (a) edge node [swap]{$1$} (d);
    \path (a) edge node {$2$} (e);
    \path (b) edge [out=205,in=245,loop,min distance=0.5cm]node [swap]{$1$} (b);
    \path (b) edge node {$0$} (f);
    \path (b) edge node [swap] {$2$} (g);
    \path (c) edge [out=-65,in=-25,loop,min distance=0.5cm]node [swap]{$2$} (c);
    \path (c) edge node {$1$} (i);
    \path (c) edge node [swap]{$0$} (h);
    \path (d) edge node [swap] {$1$} (b);
    \path (d) edge [bend right=15] node [swap]{$0$} (f);
    \path (d) edge node[swap] {$2$} (g);
    \path (e) edge [bend left=100,min distance=4.45cm] node {$1$} (i);
    \path (e) edge [bend right=15] node [swap] {$0$} (h);
    \path (e) edge  node {$2$} (c);
    \path (f) edge node {$0$} (a);
    \path (f) edge [bend right=15] node [swap]{$1$} (d);
    \path (f) edge [bend left=100,min distance=4.45cm] node {$2$} (e);
    \path (g) edge node[swap] {$0$} (h);
    \path (g) edge node [swap] {$2$} (c);
    \path (g) edge [bend right=15] node [swap] {$1$} (i);
    \path (h) edge [bend right=15] node [swap] {$2$} (e);
    \path (h) edge node [swap] {$0$} (a);
    \path (h) edge node [swap] {$1$} (d);
    \path (i) edge [bend left=100,min distance=4.45cm] node {$0$} (f);
    \path (i) edge [bend right=15] node [swap] {$2$} (g);
    \path (i) edge node {$1$} (b);
\end{tikzpicture}
\end{center}
\caption{The de Bruijn graph $G(3,2)$.\label{fig-DB-3-2-straight}}
\end{figure}

Observe that the de Bruijn graph $G(n,m)$ describes an automaton over the
alphabet $X_n$. Moreover, this automaton is synchronizing at level $m$: when
it reads the string $b_0b_1\ldots b_{m-1}$ from any initial state, it moves
into the state labeled ${b_0b_1\ldots b_{m-1}}$.

The de Bruijn graph is, in a sense we now describe, the universal automaton
over $X_n$ which is synchronizing at level $m$.

We define a \emph{folding} of an automaton $A$ over the alphabet $X_n$ to be
an equivalence relation $\equiv$ on the state set of $A$ with the property
that, if $a\equiv a'$ and $\pi_A(x,a)=b$, $\pi_A(x,a')=b'$, then $b\equiv b'$.
That is, reading the same letter from equivalent states takes the automaton
to equivalent states. If $\equiv$ is a folding of $A$, then we can uniquely
define the \emph{folded automaton} $A/{\equiv}$: the state set is the set
of $\equiv$-classes of states of $A$; and, denoting the $\equiv$-class of
$a$ by $[a]$, we have $\pi_{A/{\equiv}}(x,[a])=[\pi_A(x,a)]$ (note that this
is well-defined).

\begin{prop}
The following are equivalent for an automaton $A$ on the alphabet $X_n$:
\begin{itemize}\itemsep0pt
\item $A$ is synchronizing at level $m$, and is core;
\item $A$ is the folded automaton from a folding of the de Bruijn graph $G(n,m)$.
\label{p:fold}
\end{itemize}
\end{prop}

\begin{proof}
The ``if'' statement is clear. So suppose that $A$ is synchronizing
at level $m$. Define a relation $\equiv$ on the vertex set $X_n^m$ of $G(n,m)$
by the rule that $a\equiv b$ if the states of $A$ after reading $a$ and $b$
respectively are equal. (These states are independent of the initial state,
by assumption.) It is readily seen that $\equiv$ is a folding of $G(n,m)$,
and the $\equiv$-classes are bijective with the states of $A$. (The fact
that $A$ is core shows that the map which takes the state $q$ of $A$ to the set of $\equiv$-classes of $m$-tuples which bring $A$ to state $q$ is well-defined and injective by definition of $\equiv$, and is onto since $A$ is core.) Moreover, this bijection is clearly an isomorphism.\qed
\end{proof}

\begin{prerk} An automaton $A$ over an alphabet $X_n$ can be regarded, in
terms of universal algebra, as an algebra with unary operators $\nu_x$ for
$x\in X_n$, where the elements of the algebra are the states, and
$a\nu_x=\pi(x,a)$. A folding is precisely the kernel of an algebra homomorphism,
and the folded automaton is isomorophic to the image of the homomorphism.
The automata which are synchronizing at level $m$ form a variety, defined by
the identities
\[a\nu_{x_0}\nu_{x_1}\cdots\nu_{x_{m-1}}
=b\nu_{x_0}\nu_{x_1}\cdots\nu_{x_{m-1}}\]
for all elements $a,b$ of the algebra and all choices of $x_0,\ldots,x_{m-1}$.
\end{prerk}

We now describe how to make the de Bruijn automaton into a transducer by
specifying outputs. Let $f\in F(X_n,m+1)$ be a function from $X_n^{m+1}$ to
$X_n$. The output function of the transducer $T_f$ will be given by 
\[\lambda_T(x,a_{m-1}a_{m-2}\ldots a_{0})=(a_{m-1}a_{m-2}\ldots a_{0}x)f.\]

In other words, if the transducer reads $m+1$ symbols, then its output is
obtained by applying $f$ to the sequence of symbols read. Note that this
transducer is synchronous; it writes one symbol for each symbol read. When
applied to $x\in\xnz$, it produces $y=(x)f_\infty\in\xnz$. Recall that the function $f \in F(X_{n}, 2)$ given by $(x_{-1}x_0)f = x_{-1}$ for all $x_{-1},x_0 \in X_n$, induces the shift map $\shift{n}$ on $\xnz$ and $\xnN$. For this map we have $T_{f} = \Shift{n}$.

\begin{prerk}\label{bijectionfromFinftytoPn}
Given any de Bruijn graph $G(n,m)$, and any transducer $T$ with underlying
directed graph $G(n,m)$ there is a function $f \in F(X_n, m+1)$ such that
$T_f = T$.
\end{prerk}

Clearly the transducer $T_f$ is synchronizing at level $m$. This remains true 
if we minimise it or identify its $\omega$-equivalent states; so by Proposition~\ref{p:fold}, the
resulting minimal or weakly minimal transducer is a folding of the de Brujin graph $G(n,m)$. Let $T \in \spn{n}$ be the weakly-minimal representative of $T_{f}$, then $f_{T} = f_{T_f} = f_{\infty}$ holds since identifying $\omega$-equivalent states does not affect the map $f_{T}$. 

\begin{prerk}
The preceding paragraph together with  Remarks \ref{F(X_n,m)containedinF(X_n,m+1)} and \ref{bijectionfromFinftytoPn} show that there is a bijection from $F_{\infty}(X_n)$ to $\spn{n}$. The next result demonstrates that this bijection is a monoid homomorphism.
\end{prerk}
 
\begin{prop}\label{prop:homomorphismfromPntoFinfinitiy}
Let $A,B\in\spn{n}$. Then $ f_A \circ f_B = f_{A\ast B}$.
\end{prop}

\begin{proof}
Let $j, k $ be natural numbers such that $A$ is synchronizing at level $j$ and $B$ is synchronizing at level $k$. By Proposition
\ref{p:synchlengthsadd}, $A \ast B$ is synchronizing at level $k+j$. 

Let $x\in\xnz$ and $i \in \mathbb{Z}$ be arbitrary and $y,z,t \in \xnz$ be such that $y = (x)f_A$, $z = (y)f_B$ and $t = (x)f_{A\ast B}$. Set  $a:= x_{i-j-k}\ldots x_{i-1} \in X_{n}^{j+k}$, $b:= x_{i-k}\ldots x_{i-1} \in X_{n}^{k}$, $b'= x_{i-j}\ldots x_{i-1} \in X_{n}^{j} $ and
$c:= x_{i-j-k} \ldots x_{i-k-1} \in X_{n}^{j}$.   

By definition of the function $f_{A}$, the block  $d:= y_{i-k}y_{i-k+1}\ldots y_{i-1}$  of $y$ satisfies $dy_i$ is precisely equal to
${\lambda_A(\rev{bx_i}, q_{{c}})}$. Once more, by definition, $z_i = \lambda_B(y_i,p_{{d}})$ and since $y_i=\lambda_{A}(x_i, q_{{b'}})$,  $z_i = \lambda_{A\ast B}(x_i,(q_{{b'}}, p_{{d}}))$ as well. However,
the state of $A\ast B$ forced by ${a}$ is precisely $(q_{{b'}}, p_{{d}})$, and so we conclude that $t_i = \lambda_{A\ast B}(x_i,(q_{b'}, p_{d})) = z_i$. Since $i$ and $x$  were arbitrarily chosen, $t=z$ and $f_{A} \circ f_{B} = f_{A\ast B}$. \qed
\end{proof}

\begin{cor}\label{cor:FinftyisoPn}
	The monoid $F_{\infty}(X_{n})$ is isomorphic to $\spn{n}$.
\end{cor}

Let $\shn{n}$ be the submonoid of $\spn{n}$ consisting of those
elements $P\in\mathcal{P}_n$ all of whose states are homeomorphism states. Set $\hn{n}$ to be the largest inverse closed subset of $\shn{n}$ (where we take the automata theoretic inverse in this case). Observe that $\hn{n}$ is a group and by Proposition~\ref{prop:homomorphismfromPntoFinfinitiy} the automata theoretic inverse of  $\hn{n}$ coincides with its inverse in $\spn{n}$ as a map of $\xnz$. Thus, $\hn{n}$, as a set, is precisely the set of  core, synchronous, invertible, minimal, bi-synchronizing transducers. It is a result in \cite{AutGnr} that $\hn{n}$ is isomorphic to a subgroup of $\out{G_{n,1}}$ which in turn is a subgroup of $\out{G_{n,r}}$ for all $1\leq r<n$.  We further remark that right permutive maps
$f \in F(X_n, m)$ give rise to transducers $T_{f}$ which are elements of $\shn{n}$. By Theorem~\ref{t:hed2} we therefore have the following corollary:

\begin{theorem}
$RF_{\infty} \cong \shn{n}$ and  $Aut(\xnN, \shift{n}) \cong \mathcal{H}_n$. Thus $Aut(\xnN, \shift{n})$ is isomorphic to a subgroup of $\out{G_{n,r}}$.
\end{theorem}

\section{Automorphisms of de Bruijn graphs and $\hn{n}$}\label{sec:FiniteSubgroups}

In this section we show that a finite subgroup $G$  of $\hn{n} \cong \aut{\xnN, \shift{n}}$ is isomorphic to the automorphism group $\aut{\Gamma}$ of the underlying directed graph $\Gamma$ of an automaton $A$ arising from a folding of a de Bruijn graph. Moreover, for any directed graph $\Gamma$ underlying an automaton $A$ arising from a folding of a de Bruijn graph, there is a subgroup $G$ of $\hn{n}$ isomorphic to $\aut{\Gamma}$. We make use of this result and results  \cite{BoyleFranksKitchens} to characterize the group $\aut{\Gamma}$ for  $\Gamma$ the underlying directed graph of an automaton $A$ arising from a folding of a de Bruijn graph. In particular we show that the automorphism group of a de Bruijn $G(n,m)$  is precisely the symmetric group on a set of size $n$.

\subsection{Elements of $\hn{n}$ from automorphisms of directed graphs underlying folded automata}
We use the connection to de Bruijn graphs to construct elements of $\hn{n}$. Recall that an automaton $A$ may be regarded as labeled directed graph with vertex set $Q_{A}$,  and edge set $E_{A} \subset    Q_{A} \times X_{n}\ \times Q_{A}$. In this view, for vertices or states $p, q \in Q_{A}$, and a letter $x \in X_{n}$, $(p,x,q) \in E_{A}$ is an edge from $p$ to $q$ with label $x$ if and only if $\pi_{A}(x,p) = q$. Let $G_{A}$ denote the unlabeled directed graph corresponding to an automaton $A$. We may therefore consider the automorphisms of the directed graph $G_{A}$ underlying an automaton $A$. We construct elements of $\hn{n}$ from  automorphisms of $G_{A}$ where $A$ is a folded automata arising from foldings of de Bruijn graphs. Though, we do not distinguish between an automaton and the labeled directed graph it generates, we shall distinguish between an automaton $A$ and its unlabeled directed graph $G_{A}$. 

It turns out that all elements of $\hn{n}$ arising from an automorphism of the underlying graph of a folded automaton have finite order. In the paper \cite{BoyleFranksKitchens} Boyle \emph{et al.}  show that $\hn{n}$ is generated by elements of finite order, and give a generating set: the `vertex' and `simple' automorphisms. The elements in this generating set are in fact a subset of those elements of $\hn{n}$ constructed from automorphisms of folded automata that are considered here.

Let $G= (V,E, \iota, \tau)$ be a directed graph where $V$ is the set of vertices of $G$, $E$ is its set of edges, $\iota: E \to V$ is a map which returns the origin of an edge, and $\tau$ is a map that returns the terminus of an edge. An automorphism of $G$ is a map $\phi:= (\phi_{V}, \phi_{E})$ such that:
\begin{enumerate}
	\item $\phi_{V}: V \to V$ is a bijection,
	\item $\phi_{E}: E \to E$ is a bijection, and,
	\item for an edge $e \in E$, $ ((e)\iota)\phi_{V} = ((e)\phi_{E})\iota$ and $((e)\tau)\phi_{V} = ((e)\phi_{E})\tau$.  
\end{enumerate}
    In general usage, we shall suppress subscripts in the maps $\phi_{E}$ and $\phi_{V}$, the arguments determining which is meant in each case. Thus for an edge $e$ we write $(e)\phi$ for $(e)\phi_{E}$ and $((e)\iota)\phi$ for $((e)\iota)\phi_{V}$.  

Let $A$ be a folded automaton arising from a folding of a de Bruijn graph and let $\phi$ be an automorphism of the directed graph $G_{A}$ corresponding to $A$. Let $H(A,\phi)$ be a transducer with state set $Q_{H(A,\phi)}: = Q_{A}$ transition function $\pi_{H(A,\phi)} := \pi_{A}$ and output function $\lambda_{H(A,\phi)}: X_{n} \times Q_{H(A,\phi)} \to X_{n}$  defined as follows. For $x \in X_{n}$ and $p \in Q_{A}$, let $q = \pi_{A}(x, p)$ so that $(p,x,q)$ is an edge of $G_{A}$, let  $(r, y,s)$ be the image of $(p,x,q)$ under $\phi$, noting that $(p)\phi = r$ and $(q)\phi =s$, then set $\lambda_{H(A,\phi)}(x, p)= y$. 

The transducer $H(A, \phi)$ can be thought of as the result of gluing the automata $A$ to itself along the map $\phi: Q_{A} \to Q_{A}$.  That is, if $p,q \in Q_A$ and $(p,x,q)$ is an edge from $p$ to $q$ with label $x$ in $A$, and if $y$ is the label of the edge $((p,x,q))\phi$ in $A$, then the vertex $p$ is identified with the vertex $(p)\phi$, the vertex $q$ with the vertex $(q)\phi$, the input label is $x$ and the output label is the label $y$. Note that this fits in our view of a transducer as an ordered pair of automata where there is an isomorphism of the underlying graphs which associates to each edge of that graph a domain and range label. 

We make a few observations:
\begin{enumerate}
	\item For each state $q \in Q_{H(A, \phi)}$, the map $\lambda_{H(A, \phi)}(\centerdot, q): X_{n} \to X_{n}$ is a bijection. This follows from the definition of $G_{A}$: for each  $x \in X_{n}$ there is precisely one edge  of the form $((q)\phi, x, p)$ based at the vertex $(q)\phi$.  It follows that the transducer $H(A, \phi)$ is invertible.

	\item  If $A$ is synchronizing at level $k$ (and so a folding of $G(n,k)$ by Proposition~\ref{p:fold}) then both $H(A, \phi)$ and $H(A, \phi)^{-1}$ are synchronizing at level $k$ hence the minimal $\overline{H(A, \phi)}$ representative of $H(A, \phi)$ is an  element of $\hn{n}$. 
	
	\item In fact, for a state $q \in Q_{A}$, if $W_{k,q}$ is the set of words of length $k$, that force the state $q$, i.e., $$W_{k,q}:= \{ a \in X_{n}^{k}: \pi_{H(A, \phi)}(a,q) = q \},$$ then $\{\lambda_{H(A, \phi)}(a, p) \mid a \in Q_{k,q}, p \in Q_{H(A, \phi)} \}$ is equal to $W_{k, (q)\phi}$. 
	
	\item Let $A(H(A, \phi)) = (X_n, Q_{H(A, \phi)}, \pi_{H(A, \phi)})$ and $$A(H(A, \phi)^{-1}) = (X_{n}, Q_{H(A, \phi)^{-1}}, \pi_{H(A, \phi)^{-1}})$$ be the automata corresponding to $H(A, \phi)$ and $H(A, \phi)^{-1}$ when outputs are ignored. By construction  $A(H(A, \phi))=A$, and the previous two points indicate that $A(H(A, \phi)^{-1})$ is also isomorphic as an automaton to $A$ (by the map sending a state $q^{-1}$  of $H(A, \phi)^{-1}$ to the state $(q)\phi$ of $A$). 
\end{enumerate}

The third point above and results of the paper \cite{OlukoyaOrder} show that an element of $\hn{n}$ obtained from an automorphism of a folded automaton must have finite order. This result, which also follows from Theorem~\ref{thm:autfoldembedsinHn} below, means that not all elements of $\hn{n}$ for $n \ge 3$ arise from automorphisms of the directed graph underlying some folded automaton.

\subsection{Automorphisms of folded automata and permutations of the alphabet.}
Consider the de Bruijn graph $G(n,m)$. Any permutation $\rho$ of the set $X_{n}$ induces an automorphism, which we again denote by $\rho$, of $G(n,m)$ as follows. A vertex $a= a_1a_2\ldots a_{n}$ of the graph $G(n,m)$ is mapped to the vertex $b = (a_1)\rho (a_2)\rho \ldots (a_{n})\rho:= (a)\rho$. An edge $e = (a,x,b)$ is mapped to the edge $((a)\rho, (x)\rho, (b)\rho)$. In this case, the transducer $H:= H(G(n,m), \rho)$ arising from the pair $(G(n,m), \rho)$ has the property that for any state $q \in Q_{H}$, the bijection  $\lambda_{H}(\centerdot, q): X_{n} \to X_{n}$ is the permutation $\rho$. Therefore, the minimal transducer $\overline{H}$ representing $H$ has exactly one state, and this state induces the permutation $\rho$ on the alphabet  $X_{n}$. We show below that these are the only automorphisms of the automaton $G(n,m)$.

Let $A$ be a folded automaton arising from a folding of $G(n,m)$ and let $\rho$, as above, be a permutation of $X_{n}$. By the definition of a folding the individual states of $A$ correspond to subsets of the vertices of $G(n,m)$ and the set of states of $A$ forms a partition of the vertices of $G(n,m)$. As vertices of $G(n,m)$ are words of length $m$ in $X_{n}$, we may define a map $\phi_{V_{A}}$ on the vertices of $G_{A}$ to the set of subsets of $X_{n}^m$, by mapping a vertex $q$ to the set $\{ (a)\rho \mid a \in X_{n}^{m} \cap q \}$. If the image of $\phi_{V_A}$ in the set of subsets of $X_n^m$ is again precisely the partition $V_{A}$, then we may define an edge map $\phi_{E_{A}}: E_{A} \to E_{A}$ by mapping an edge $(a, x, b)$ to the edge $((a)\rho, (x)\rho, (b)\rho)$ and this will be well defined for the folding $A$ by the definition of a folding. In this case, the map $(\phi_{V_A},\phi_{E_{A}})$ is an automorphism of $G_{A}$ which we once again denote by $\rho$.

The example below indicates that, in general, not all automorphisms of the directed graph underlying a folded automaton arise from a permutation of the symbol set.

\begin{figure}[htbp]
	\begin{center}
	\begin{tikzpicture}[shorten >= .5pt,node distance=3cm,on grid,auto] 
	\node[state] (q0)   {$q_0$};
	\node[state, xshift= -2cm, yshift=-2cm] (q1)   {$q_1$};
	\node[state, xshift=2cm, yshift=-2cm ] (q2) {$q_2$};
	\node[state, xshift=7cm] (p0) {$q_0$};
	\node[state, xshift= 5cm, yshift=-2cm] (p1)   {$q_1$};
	\node[state, xshift=9cm, yshift=-2cm ] (p2) {$q_2$};
	\node[xshift=0cm, yshift=-3.5cm]{$A$};
	\node[xshift=7cm, yshift=-3.5cm]{$G_A$};
	\path[->] 
	
	(q0) edge[out=180, in=90]  node[swap]{$1$} (q1)
	     edge[out=75, in=105, loop] node[swap] {$0$}() 
	     edge[out=360, in=90] node{$2$} (q2)
	     
	(q1) edge[in=190, out=80]  node[swap] {$0$} (q0)
	edge[out=165, in=195,loop] node[swap] {$1$} ()
	edge[in=190, out=350] node[swap] {$2$} (q2) 
	
	(q2) edge[in=360, out=180] node[swap] {$0$}   (q1)
	edge[in=345, out=15,loop] node[xshift=-0.1cm] {$2$}  ()
	edge[in=350, out=100] node {$1$}  (q0)
	
	(p0) edge[out=180, in=90] (p1)
	edge[out=75, in=105, loop] () 
	edge[out=360, in=90] (p2)
	
	(p1) edge[in=190, out=80] (p0)
	edge[out=165, in=195,loop] ()
	edge[in=190, out=350] (p2) 
	
	(p2) edge[in=360, out=180] (p1)
	edge[in=345, out=15,loop] ()
	edge[in=350, out=100] (p0);
  \end{tikzpicture}
\end{center}
\caption{A folded automaton with an automorphism not induced by a permutation.\label{fig-nonpermaut}}
\end{figure}
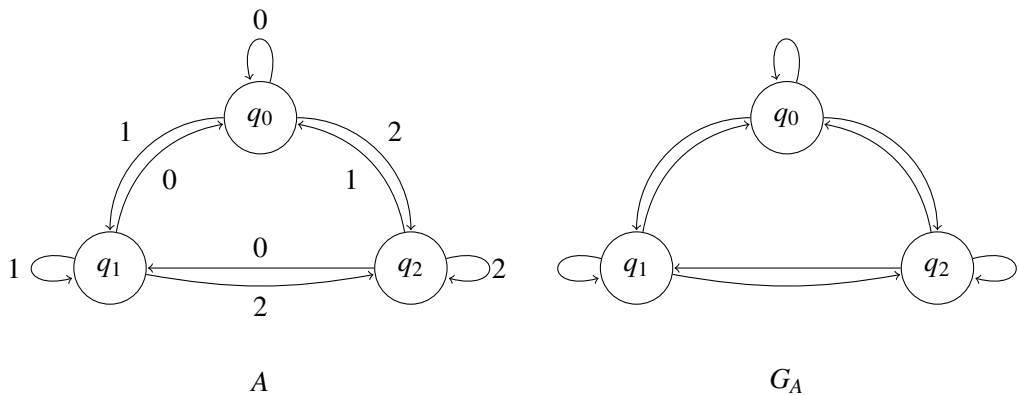

The automorphism group of the  underlying directed graph of the automaton $A$ in Figure~\ref{fig-nonpermaut} is the group $S_{3}$ as all three vertices may be permuted and any permutation of the three vertices forces a bijection on the edges. The automaton $A$ is a folded automaton arising from a folding of $G(3,2)$; the vertex $q_0$ corresponds to the set $\{00,21,10\}$, the vertex $q_1$ corresponds to the set $\{01,11,20\}$ and the vertex $q_2$ to the set $\{02,12,22\}$. The automorphism $\phi$ which swaps the vertex $q_0$ with $q_2$ but fixes the vertex $q_1$ is not induced by a permutation of the set $X_3$. (If $\phi$ were induced by a permutation $\rho$ of $X_n$, then $\{(00)\rho, (21)\rho,(10)\rho \} = \{02,12,22\}$ and $ \{(01)\rho,(11)\rho,(20)\rho\}=\{01,11,20\}$, which is not possible.)

The result below characterizes when an automorphism of the directed graph of a folded automaton is induced by a permutation of the alphabet set $X_{n}$.

\begin{prop}\label{prop:permaut}
	Let $A$ be a folded automaton arising from a folding of $G(n,m)$. An automorphism $\phi$ of the graph $G_A$ arises from a permutation $\rho$ of the set $X_{n}$ if and only if the minimal representative  of the transducer $H:=H(A,\phi)$ has exactly one state, and this state induces the permutation $\rho$ on $X_{n}$.
\end{prop}
\begin{proof}
	If the automorphism $\phi$ arises from a permutation $\rho$ of $X_{n}$, then as in the $G(n,m)$ case, all state of the transducer $H$ induce the permutation $\rho$ on the set $X_{n}$. Therefore, the minimal representative $\overline{H}$ of $H$ has exactly one state, and this state induces the permutation $\rho$ on $X_{n}$.
	
	Therefore, suppose that the minimal representative $\overline{H}$ of  the transducer $H = H({A}, \phi)$ has exactly one state, and this state induces the permutation $\rho$ on $X_{n}$. It must be the case that all states of $H$ induce the permutation $\rho$ on $X_{n}$. It follows that for an edge $e = (p,x,q)$ of $G_{A}$, $(e)\phi = ((p)\phi, (x)\rho, (q)\phi)$. Let $q$ be state of $H$, then as, by definition, $q$ is a state of $A$ $q$ corresponds to a subset of $X_{n}^{m}$. In particular, $q$ corresponds to the subset $W_{m,q}$ of $X_{n}^{m}$ consisting of all elements of $X_{n}^{m}$ which force the state $q$ when read from any state of $A$. Now as all states of $H$ induce the permutation $\rho$ on $X_{n}$, it follows that the state $q^{-1}$ of the automaton $H^{-1}$ corresponds to the subset $\{ (a)\rho \mid a \in W_{m,q}\}$. Therefore as $(q)\phi = q^{-1}$, we see that $\phi$ must arise from the permutation $\rho$. \qed
\end{proof}

Returning to the automaton $A$ in Figure~\ref{fig-nonpermaut}, the automorphism $\phi$ of $G_{A}$ which swaps the vertices $q_0$ and $q_1$ yields the automaton $(A)\phi$ and the transducer $H(A,\phi)$ depicted in Figure~\ref{fig-transducernonpermaut}. The transducer $H(A,\phi)$ is minimal.

\begin{figure}[htbp]
	\begin{center}
		\begin{tikzpicture}[shorten >= .5pt,node distance=3cm,on grid,auto] 
		\node[state] (q0)   {$q_2$};
		\node[state, xshift= -2cm, yshift=-2cm] (q1)   {$q_1$};
		\node[state, xshift=2cm, yshift=-2cm ] (q2) {$q_0$};
		\node[state, xshift=7cm] (p0) {$q_0$};
		\node[state, xshift= 5cm, yshift=-2cm] (p1)   {$q_1$};
		\node[state, xshift=9cm, yshift=-2cm ] (p2) {$q_2$};
		\node[xshift=0cm, yshift=-3.5cm]{$A$};
		\node[xshift=7cm, yshift=-3.5cm]{$H(A, \phi)$};
		\path[->] 
		
		(q0) edge[out=180, in=90]  node[swap]{$1$} (q1)
		edge[out=75, in=105, loop] node[swap] {$0$}() 
		edge[out=360, in=90] node{$2$} (q2)
		
		(q1) edge[in=190, out=80]  node[swap] {$0$} (q0)
		edge[out=165, in=195,loop] node[swap] {$1$} ()
		edge[in=190, out=350] node[swap] {$2$} (q2) 
		
		(q2) edge[in=360, out=180] node[swap] {$0$}   (q1)
		edge[in=345, out=15,loop] node[xshift=-0.1cm] {$2$}  ()
		edge[in=350, out=100] node {$1$}  (q0)
		
		(p0) edge[out=180, in=90] node[swap]{$1|0$} (p1)
		edge[out=75, in=105, loop] node[swap] {$0|2$} () 
		edge[out=360, in=90] node{$2|1$} (p2)
		
		(p1) edge[in=190, out=80]  node[swap] {$0|2$} (p0)
		edge[out=140, in=170,loop] node[swap] {$1|1$} ()
		edge[in=190, out=350] node[swap] {$2|0$} (p2) 
		
		(p2) edge[in=360, out=180] node[swap] {$0|1$} (p1)
		edge[out=40, in=10,loop] node {$2|0$} ()
		edge[in=350, out=100] node {$1|2$}   (p0);
		\end{tikzpicture}
	\end{center}
	\caption{The transducer arising from the automorphism swapping vertices $q_0$ and $q_1$.\label{fig-transducernonpermaut}}
\end{figure}
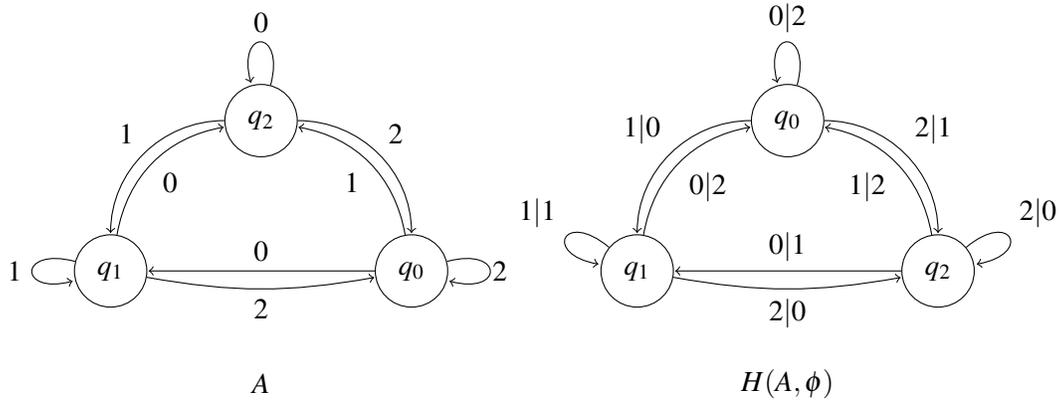

\begin{theorem}\label{thm:autfoldembedsinHn}
	Let $A$ be a  folded automaton arising from a folding of $G(n,m)$ for $m$ minimal. The map from the group $\aut{G_{A}}$ of automorphisms of the directed graph $G_{A}$ to $\hn{n}$ which maps an automorphism $\phi$  to the minimal representative of the transducer $H(A, \phi)$, is a monomorphism.
\end{theorem}
\begin{proof}
    If $|A| = 1$ then the result is a consequence of Proposition~\ref{prop:onestatecond}. Thus we may assume that $|A|>1$.
    
    Let $\phi$ be a non-trivial automorphism of $G_{A}$. Then as $\phi$ is not trivial, either it moves some state or fixes every state and move some edges. 
    
    Suppose firstly that $\phi$ moves some state. Let $p,q \in Q_A$ be distinct states that $(p)\phi = q$. Since, $A$ is a folding of $G(n,m)$, $p$ and $q$ correspond to distinct subsets  of $X_{n}^{m}$ consisting of all words $W_{m,p}$ and $W_{m,q}$ that force the states $p$ and $q$ respectively. Now, by an observation above, the state $p$ of $H(A,\phi)$ is such that $\lambda_{H(A,\phi)}(\centerdot,p)$ induces a bijection from $W_{m,p}$ to $W_{m,q}$. Therefore, we see that $H(A,\phi)$ is not the identity transducer.
    
    In the case that $\phi$ fixes every state and moves some edges, let $e = (p,x,q)$ be an edge move by $\phi$. Since $\phi$ fixes all vertices, there must be an edge $(p, y,q)$ from $p$ to $q$, for $x \ne y$ such that $((p,x,q))\phi = (p,y,q)$. In this case, we have that the state $p$ of $H(A, \phi)$ satisfies $\lambda_{A}(x,p) = y$. We once again conclude that $H(A,\phi)$ is not the identity transducer.
    
    Therefore, the only element of $\aut{G_{A}}$ that maps to the identity transducer, is the identity element. This means that it suffices to show that the map from $\aut{G_A} \to \hn{n}$ which sends an automorphism $\phi$ to the minimal representative of  $H(A, \phi)$ is a homomorphism to conclude that it is a monomorphism.
    
    Let $\phi, \psi$ be two automorphisms of $G_{A}$ and let $H(A,\phi)$ and $H(A,\psi)$ be the corresponding transducers. Notice that the trio $H(A,\phi)$, $H(A, \psi)$ and  $H(A, \phi\psi)$, all by definition, have state set $Q_{A}$. This should not cause confusion below, as whenever we write a pair $(p,q)$ $H(A,\phi) \ast H(A,\psi)$, the first coordinate corresponds to the state of $H(A, \phi)$ and the second to the state of $H(A,\psi)$ and for   a single state  $p \in Q_A$ it will be clear below which of the three transducers $H(A,\phi)$, $H(A, \psi)$ and  $H(A, \phi\psi)$ it is being regarded as a state of. On the other hand, the set $W_{m,q}$  for a state $q \in Q_{A}$, depends only on the automaton $A$. That is the set of words in $X_{n}^{m}$ which force the state $q$ in  $H(A,\phi)$,  $H(A, \psi)$ or $H(A, \phi\psi)$  are all equal to $W_{m,q}$. 
    
    A state $(p,q)$ of the product $H(A,\phi) \ast H(A,\psi)$ is a state of the core if and only if $\{ a \in X_{n}^{m} \mid a = \lambda_{H(A,\phi)}(b, q) \mbox{ for some } b \in W_{m,p}, q \in Q_{A}\} = W_{m,q}$. This is because, by an observation above,  $$\{ a \in X_{n}^{m} \mid a = \lambda_{H(A,\phi)}(b, q) \mbox{ for some } b \in W_{m,p}, q \in Q_{A}\} = W_{m, (p)\phi}$$ and this set depends only on $A$. Thus a state $(p,q)$ is a state of the $\core(H(A,\phi) \ast H(A,\psi))$ if and only if it is of the form $(p, (p)\phi)$.
    
    Let  $(p,x,q)$ be an edge of $G_{A}$, $((p)\phi, y, (q)\phi)$ be its image under $\phi$ and $((p)\phi\psi, z, (q)\phi\psi)$ its image under $\phi\psi$. This means that the state $p$ of $H(A,\phi)$ satisfies, $\lambda_{H(A,\phi)}(x,p) = y$ and $\pi_{H(A,\phi)}(x, p) = q$. The state $(p)\phi$ of $H(A,\psi)$ satisfies, $\lambda_{H(A,\psi)}(y, (p)\phi) = z$ and $\pi_{H(A,\psi)}(y, (p)\phi) = (q)\phi$. Thus $$\lambda_{H(A,\phi\psi)}(x, (p,(p)\phi)) = z$$ and $$\pi_{H(A,\phi\psi)}(x, (p,(p)\phi)) = (q, (q)\phi).$$ 
    
    The above calculation demonstrates that the  map from $H(A, \phi\psi)$ to $\core(H(A,\phi) \ast H(A,\psi))$ which sends a state $p$ of $H(A, \phi\psi)$ to the state $(p, (p)\phi)$ of $\core(H(A,\phi) \ast H(A,\psi))$ is an automaton isomorphism. This concludes the proof. \qed   
\end{proof}

\subsection{Finite subgroups of $\hn{n}$}
We observe that a converse of Theorem~\ref{thm:autfoldembedsinHn} is valid, namely, every finite subgroup of $\hn{n} \cong \aut{\xnn, \shift{n}}$ arises from the automorphism group of a folded de Bruijn graph. This follows from work in the paper \cite{BoyleFranksKitchens}, however  we give a proof below.



The other construction we require is the \emph{dual automaton} (see \cite{AkhaviKlimannLombardyMairessePicantin2012, NekrashevychSSG}). 

Let $T$ be a transducer over the alphabet $\xn$. Set $\dual{T} = \gen{Q_{T}, X_{n}, \dpi_{T}, \dlambda_{T}}$, that is the state set of $\dual{T}$ is the set $X_{n}$, the alphabet of $\dual{T}$ is the state set $Q_{T}$ of $T$, and the transition $\dpi_{T}$ and output functions $\dlambda_{T}$ are defined as follows. For $q \in Q_{T}$ and $x \in \xn$, $\dpi_{{T}}(q, x) = y$ and $\dlambda_{{T}}(q, x) = p$ if and only if $\pi_{T}(x,q) = p$ and $\lambda_{T}(x, q) = y$.

One can easily check the following lemma.

\begin{lemma}
    Let $T$ be a synchronous transducer over alphabet $X_n$.  For positive natural $m$, we have $(\dual{T})^{m}=\dual{T(m)}$.
\end{lemma} 
Note that to lighten our notation below, we may use the notation $\dual{T}_{m}$ for the transducer $\dual{T(m)}$.

Also observe that $\dual{T^{-1}}$ is obtained from $\dual{T}$ by `reversing the arrows'. That is if, $x,y \in \xn$, $q, p \in Q_{T}$ are such that $\dpi_{T}(q, x) = y$ and $\dlambda(q, x) =p$, then $\dpi_{T^{-1}}(q^{-1}, y) = x$ and  $\dlambda(q^{-1}, y) =p^{-1}$.



The proof we give below is more automata theoretic and is based on the following result from \cite{OlukoyaOrder}.

\begin{prop}\label{prop:dualzero}
    Let $G \le \hn{n}$ be a finite subgroup. Let $k \in  \N$ the largest minimal synchronizing level of any element of $G$. Then for any $H \in G$, and for any word $\Gamma \in \xn^{k}$, there is a word $W(\Gamma,H) \in Q_{H}^{+}$ such that for any word $P \in Q_{H}^{+}$, $\dpi_H(P,\Gamma) = W(\Gamma, H)^{i}\overline{W(\Gamma,H)}_{r}$, where, $i \in \N$, satisfies,  $|P| = i|W(\Gamma,H)| + r$, for $1 \le r < |W(\Gamma,H)|$ and $\overline{W(\Gamma,H)}_{r}$ is the length $r$ prefix of $W(\Gamma,H)$.
\end{prop}

\begin{theorem}
 Let $G \le  \hn{n}$ be a finite subgroup, then $G$ is isomorphic to a subgroup of the automorphism group of the underlying digraph of a strongly synchronizing automaton $A(G)$. Moreover, every element of $G$ is the minimal representative of a transducer $H(A(G), \phi)$ for an automorphism $\phi$ of the underlying di-graph of $A(G)$.
\end{theorem}
\begin{proof}
 Let $k \in \N$ be the such that any element of $G$ has minimal synchronizing level at most $k$. Define an equivalence relation $\sim$ on $\xn^{k}$ as follows: $\Gamma \sim \Delta$ if and only if $W(\Gamma, H) = W(\Delta, H)$ for all $H \in G$.
 
 Observe that, for $\Gamma = a \gamma$ and $\Delta = d \delta$, for $a, d \in \xn$, in the same equivalence class, then for $x \in \xn$, $\gamma x$ and $\delta x$ are also in the same equivalence class. This is because for any $H \in G$ and any word $P \in Q_{H}^{+}$ we have, $\lambda_{H^{|P|}}(a\gamma,P) = \lambda_{H^{|P|}}(a\delta, P)$, and so $\lambda_{H^{|P|}}(a\gamma x, P) = \lambda_{H^{|P|}}(a\delta x, P) $. From this we deduce that $W(\gamma x, H) = W(\delta x, H)$. 
 
 Thus, writing $[\gamma]$ for the equivalence class  of an element $\gamma$ of $\xn^{k}$, we may form an automaton $A(G)$ with state set $\xn^{k}/\sim$, and transitions $\pi_{A(G)}(x, [\gamma]) = [\overline{\gamma} x]$ where $\overline{\gamma}$ is the length $|\gamma| -1$ suffix of $\gamma$. By the previous a paragraph the automaton $A(G)$ is well defined; by construction the automaton $A(G)$ is strongly synchronizing.
 
 We now show that $G$ acts by automorphisms on the underlying digraph  of  $A(G)$.
 
 We begin by proving  the following observation. Let $\gamma, \delta \in \xn^{k}$ belong to the same equivalence class, and let $H \in G$ be arbitrary. Then for any $p, q \in  Q_{H}$, the elements of the set $\{\lambda_{H}(\xi, t) \mid  (\xi,t) \in \{(\gamma,p), (\delta,q)\}\}$ belong to the same equivalence class. 
 
 First observe that by Proposition~\ref{prop:dualzero}, there is a word $W_{H}\in Q_{H}^{+}$ such that $W_H = W(\lambda_{H}(\xi, t),H)$ for all $(\xi,t) \in \{ (\gamma,p), (\delta,q)\}$. Since $\gamma$ and $\delta$ are in  the same equivalence class, let $s_0$ be the state of $H$ forced by both $\gamma$ and $\delta$. Let  $I \in G$, $I \ne H$ be arbitrary, we show that there is a word $W_{I} \in Q_{I}^{+}$ such that $W_I = W(\lambda_{H}(\xi, t),I)$ for all $\xi \in \{ \gamma, \delta\}$ and all $t \in \{p,q\}$.  We prove this inductively.
 
  Let us establish the base case. Observe that since $HI \in G$ and since $\gamma$ and $\delta$ are in the same equivalence class, there is a unique state, $s_{1}$ of $HI$ such that for any state $s \in Q_{HI}$, the state of $HI$ forced by $\lambda_{HI}(\gamma, s)$ and $\lambda_{HI}(\delta, s)$ are equal and are equal to $s_{1}$. Notice that $HI$ is the minimal representative of $\core (H \ast I)$. There are state $s, s' \in I$ such that $(p,s), (q,s')$ are states of $\core(H \ast I)$; let $t, t' \in Q_{I}$ be such that $\pi_{H \ast I}(\gamma,(p,s)) = (s_0, t) $ and $\pi_{H \ast I}(\delta, (q,s')) = (s_0, t')$. Since the state of $HI$ forced by $\gamma$ and $\delta$ is $s_1$, we have $(s_0, t)$ and $(s_0, t')$ are $\omega$-equivalent to the state $s_{1}$, and so  $t = t'$. Set $t_1 = t = t'$. Therefore we  have shown that the state of $I$ forced by $\lambda_{H}(\gamma, p)$ is equal to the state of $I$ forced by $\lambda_{H}(\delta, q)$ and  that state is $t_1$. 
  
  Inductively assume that there is an $m \in \N$ such that for any word $u \in Q_{I}^{+}$ of length $m$, $\pi_{I^{m}}(\lambda_{H}(\gamma, p), u) = \pi_{I^{m}}(\lambda_{H}(\delta, q), u) = t_1t_2\ldots t_{m} $. We now prove the inductive step.
  
  As before, $HI^{m+1}$ is an element of $G$ and, as  $\gamma$ and $\delta$ are in the same equivalence class, they both force the same state $s_{m+1}$ of $HI^{m+1}$. There are words $s, s' \in Q_{I}^{m+1}$ such that $ps$ and $qs'$ are states of $\core(H \ast \underbrace{I \ast I \ldots \ast I}_{\text{$m+1$ times}})$. Since $HI^{m+1}$ is the minimal representative of $\core(H \ast \underbrace{I \ast I \ldots \ast I}_{\text{$m+1$ times}} )$, it follows that, if $T_{m+1}, T'_{m+1} \in Q_{I}^{m+1}$ satisfy, $\pi_{H \ast I}(\gamma,ps) = s_0 T_{m+1} $ and $\pi_{H \ast I}(\delta, qs')) = s_0T'_{m+1}$, then $s_0T_{m+1}$ and $s_0T'_{m+1}$ are both $\omega$-equivalent to the state $s_{m+1}$ of $HI^{m+1}$. By the inductive assumption, we have that that the first $m$ letters  of $T_{m+1}$ and $T'_{m+1}$ coincide, the preceding sentence now implies that $T_{m+1} = T'_{m+1}$. Set $t_{m+1}$ to the final letter of $T_{m+1} = T'_{m+1}$. By Proposition~\ref{prop:dualzero} it now follows that for any word for any word $u \in Q_{I}^{+}$ of length $m+1$, $\pi_{I^{m+1}}(\lambda_{H}(\gamma, p), u) = \pi_{I^{m+1}}(\lambda_{H}(\delta, q), u) = t_1t_2\ldots t_{m}t_{m+1} $. We therefore conclude that there is a word $W_{I} \in Q_{I}^{+}$ such that $W_I = W(\lambda_{H}(\xi, t),I)$ for all $\xi \in \{ \gamma, \delta\}$ and all $t \in \{p,q\}$. 
  
  Since $I  \in G$, $I \ne H$, was chosen arbitrarily, it follows that  $\lambda_{H}(\gamma, p)$ and $\lambda_{H}(\delta, q)$ are  in the same equivalence class.
 
 Let $[\gamma]$ be a vertex of $A(G)$, let $\overline{\gamma}$ be the length $k-1$ suffix of $\gamma$ and let $x \in \xn$ be the label of the edge from $[\gamma]$ to $[\overline{\gamma} x]$. Let $H \in G$ be arbitrary  and let $y = \lambda_{H}(x, q_{\gamma})$, then by the preceding paragraph for any pair of states $p, q \in Q_{H}$ and any $I \in G$, $W(\lambda_{H}(\gamma,p)y, I) = W(\lambda_{H}(\gamma,q)y, I)$. From this it follows that setting $\mu, \nu$ to be the length $k-1$ suffices of $\lambda_{H}(\gamma,p)$  and $\lambda_{H}(\gamma,q)$ respectively, $[\mu y] = [\nu y]$. Now as there is a state $s$ of $H$ such that $\lambda_{H}(\overline{\gamma} x,s) = \mu y$, it follows, by the preceding paragraphs once more, that for any state $t \in Q_{H}$, $[\lambda_{H}(\overline{\gamma} x,t)] = [\mu y]$. Since $\mu$ is a length $k-1$ suffix of an element of $[\lambda_{H}(\gamma, p)]$, there is an edge labeled $y$ from $[\lambda_{H}(\gamma, p)]$ to $[\mu y]$. 
 
 For $H \in G$, define a map $\phi_{H}$ as follows. For a vertex $[\gamma]$, and edge labeled $x$ from $[\gamma]$ to $[\overline{\gamma}x]$ of the digraph $A(G)$ (where $\overline{\gamma}$ is the length $k-1$ suffix of $\gamma$) of $A(G)$, $([\gamma])\phi_{H} = [\lambda_{H}(\gamma, p)]$, $([\overline{\gamma}x) \phi_{H} = [\lambda_{H}(\overline{\gamma}x, p)]$, for some state $p \in Q_{H}$, and the edge $x$ maps to the edge labeled $\lambda_{H}(x, q_{\gamma})$ from the state $[\lambda_{H}(\gamma, p)]$ to the state $[\lambda_{H}(\overline{\gamma}x, p)]$. By the preceding a paragraphs this map is well defined. It is easily verified that for $H, I \in G$, $\phi_{HI} = \phi_{H} \phi_{I}$. Thus the map $H \mapsto \phi_H$ is an embedding of $G$  into the automorphism group of the underlying digraph of $A(G)$.    Moreover, it is not hard to see that the minimal representative of the tranducer $H(A(G), \phi_{H})$ is $H$. \qed 
 \end{proof}

In light of Theorem~\ref{thm:autfoldembedsinHn} above,   Theorem~3.8 of \cite{BoyleFranksKitchens} can be states as follows:

\begin{cor}\label{cor:characterizationofautfold}
		Let $A$ be a  folded automaton arising from a folding of $G(n,m)$ for $m$ minimal. For the group $\aut{G_{A}}$ of automorphisms of the directed graph $G_{A}$, one of the following holds:
		
		\begin{enumerate}[label=(\roman*)]
			\item $\aut{G_{A}}$ isomorphic to a subgroup of $\sym{\xn}$ that has a composition factor that cannot be embedded in $\sym{X_{n-1}}$. In this case all automorphisms of $G_{A}$ arise as permutations of the symbol set $\xn$.
				
		    \item All the composition factors of $\aut{G_{A}}$ are isomorphic to subgroups of $\sym{X_{n-1}}$. 
	 \end{enumerate}
\end{cor}

\begin{cor}
	Let $A$ be a folded automaton arising from a folding of $G(3,m)$ for some $m \in \N$. The group $\aut{G_{A}}$ is either $\sym{X_{3}}$ or a $2$-group.
\end{cor}

It is a result of Hedlund \cite{Hedlund69} that $\aut{X_{2}, \shift{2}}$ is isomorphic to the cyclic group of order $2$. Below we give a new proof of this result by identifying conditions on (non-minimal) strongly synchronizing transducers to have a  minimal representative in $\hn{n}$ with exactly one state. From this we also derive implications (via Proposition~\ref{prop:permaut}) for folded automata: more precisely we show that certain folded automata, including the graphs $G(n,m)$, admit only automorphisms arising from permutations of the symbol set $X_{n}$. 
\subsection{Synchronizing sequences}
We require an algorithm given in \cite{AutGnr} for detecting when an automaton is  strongly synchronizing. We state a version below.

Let $A = (X_{n}, Q_{A}, \pi_{A})$ be an automaton. Define an equivalence relation $\sim_{A}$ on the states of $A$ by $p \sim_{A} q$ if and only if the maps $\pi_{A}(\cdot, p): Q_{A} \to Q_{A}$ and $\pi_{A}(\cdot, q): Q_{A} \to Q_{A}$ are equal. For a state  $q \in Q_{A}$ let $\mathsf{q}$ represent the equivalence class of $q$ under $\sim_{A}$. Further set $\mathsf{Q}_{\mathsf{A}}:= \{ \mathsf{q} \mid q \in Q_{A}\}$ and let ${\pi}_{\mathsf{A}}: \mathsf{Q}_{\mathsf{A}} \to \mathsf{Q}_{\mathsf{A}}$ be defined by ${\pi}_{\mathsf{A}}(x, \mathsf{q}) = \mathsf{p}$ where $p = \pi_{A}(x,q)$. Observe that $\pi_{\mathsf{A}}$ is a well defined map. Define a new automaton  $\mathsf{A} = (X_{n}, \mathsf{Q}_{\mathsf{A}}, \pi_{\mathsf{A}})$ noting that $|\mathsf{Q}_{\mathsf{A}}| \le |Q_{A}|$ and $|\mathsf{Q}_{\mathsf{A}}| = |Q_{A}|$ implies that $A$ is isomorphic to $\mathsf{A}$.

Given an automaton $A$, let $A_{0}:=A, A_1, A_2, \ldots$ be the sequence of automata such that $A_{i} = \mathsf{A}_{i-1}$ for all $i \ge 1$. We call the sequence $(A_i)_{i \in \mathbb{N}}$ the \emph{synchronizing sequence of $A$}. We make a few observations. 

By definition each term in the synchronizing sequence is a folding of the automaton which precedes it, therefore there is a $j \in \mathbb{N}$ such that all the $A_{i}$ for $i \ge j$ are isomorphic to one another. By a simple induction argument, for each $i$, the states of $A_{i}$ corresponds to a partition of $Q_{A}$. We identify the states of $A_{i}$ with this partition. For two states  $q,p \in Q_{A}$ that belong to a state $P$ of $A_{i}$, $\pi_{A}(x,q)$ and $\pi_{A}(x,p)$ belong to the same state of $Q_{A}$ for all $x \in X_{n}$. We will use the language `two states of $A$ are identified at level $i$' if the two named states belong to the same element of $Q_{A_{i}}$. 

If the automaton $A$ is strongly synchronizing and core, then an easy induction argument shows that all the terms in its synchronizing sequence are core and  strongly synchronizing as well (since they are all foldings of $A$). For example if $A = G(n,m)$, then the first $m$ terms of the synchronizing sequence of $A$ are $(G(n,m), G(n,m-1), G(n,m-2),\ldots, G(n,1)$, after this all the terms in the sequence are the single state automaton on $X_{n}$.

The result below is from \cite{AutGnr}.

\begin{theorem}\label{thm:collapsingprocedure}
	Let  $A$ be an automaton and $A_{0}:=A, A_1, A_2, \ldots$ be the sequence of automata such that $A_{i} = \mathsf{A}_{i-1}$ for all $i > 1$. Then 
	\begin{enumerate}
		\item a pair of states $p,q \in Q_{A}$, belong to the same element $t \in Q_{A_i}$ if and only if for all words $a \in X_{n}^{i}$, $\pi_{A}(a,p) = \pi_{A}(a, q)$, and
		
		\item $A$ is strongly synchronizing if and only if there is a $j \in \mathbb{N}$ such that $|Q_{A_{j}}| = 1$. The minimal $j$ for which $|A_{j}| = 1$ is the minimal synchronizing level of $A$.
			\end{enumerate}
	
\end{theorem}

\subsection{Applying synchronizing sequences to understand automorphisms of de Bruijn graphs}

\begin{lemma}\label{lem:letterspartitionstates}
	Let  $A$ be a core strongly synchronizing automaton, $A_0:= A, A_1, \ldots $ be its synchronizing sequence and $j \in \mathbb{N}$ be minimal such that $A_{j} = 1$. If $A_{j-1}$ is isomorphic as an automaton to $G(n,1)$ then the sets $Q_{A,x}:= \{ \pi_{A}(x, p) \mid p \in Q_{A}\}$ for $x \in X_{n}$ form a partition of the set $Q_{A}$ the states of $A$.
\end{lemma}
\begin{proof}
	This follows from the identification of the states of $A_{i}$ with partitions of states of $A$. For if there were distinct $x,y \in X_{n}$ and states $p_1, p_2 \in Q_{A}$ such that $\pi_{A}(x, p_1) = \pi_{A}(y,p_2)$, then the states $P_1$ and $P_2$ of $A_{j-1}$ containing $p_1$ and $p_2$ respectively satisfy, $\pi_{A_i}(x, P_1) = \pi_{A_{i}}(y, P_2)$. However, since $A_{j-1}$ is isomorphic as an automaton to $G(n,1)$ this is not possible ($A_{j-1}$ has $n$ distinct states, is synchronizing at level $1$ and core).\qed
\end{proof}

A consequence of Lemma~\ref{lem:letterspartitionstates} is the following result.

\begin{lemma}\label{lem:nomin}
	There is no minimal, core, invertible transducer $T$ which is bi-synchronizing at minimal level $(j,k)$ and satisfies the following: if $A$ and $B$ are the automata obtained from $T$ and $T^{-1}$ respectively by forgetting outputs, then the terms $A_{j-1}$ and $B_{k-1}$  in the synchronizing sequences of $A$ and $B$ are isomorphic to $G(n,1)$.
\end{lemma}
\begin{proof}
	Since $T$ is minimal and strongly synchronizing, there is a pair $p,q \in Q_{T}$ and $x \in X_{n}$ such that $\pi_{T}(x,p = \pi_{T}(x,q)$ but $y:=\lambda_{T}(x,p) \ne \lambda_{T}(x,q)=:z$. However, we therefore have that in $T^{-1}$, and so in $B$, $\pi_{T^{-1}}(y, p^{-1}) = \pi_{T}(z,q^{-1})$ with $z \ne q$. This contradicts Lemma~\ref{lem:letterspartitionstates}.\qed
\end{proof}

We note the lack of the minimality hypothesis in the statement of the proposition below. We require the non-minimality hypothesis in order to deduce results about  elements of $\hn{n}$ arising from automorphisms of de Bruijn graphs $G(n,m)$. In particular as a consequence of the following Proposition, we show that $\aut{G_{n,m}}$ is isomorphic to the symmetric group on $n$ symbols.

\begin{prop}\label{prop:onestatecond}
	Let $T$ be a core,  invertible bi-synchronizing transducer of size at least $2$ with automata theoretic inverse $T^{-1}$. Let $(j,k)$ be the minimal bi-synchronizing level of $T$, and let  $A$ and $B$ be the automata obtained from $T$ and $T^{-1}$ respectively by forgetting outputs. Suppose that the terms $A_{j-1}$ and $B_{k-1}$ in the synchronizing sequence $(A_{i})_{i \in \mathbb{N}}$ and $(B_{i})_{i \in \mathbb{N}}$ of $A$ and $B$ are  both isomorphic, as automata, to $G(n,1)$. Then $j=k$ and the minimal transducer representing $T$ has only one state.
\end{prop}
\begin{proof}
	We proceed by induction on the number of states of $T$. 
	
	Note that as $j,k \ge 1$, it follows that the base case occurs when $|T| = n$. In this case, both  $A$ and $B$ are isomorphic to $G(n,1)$ and $j=k =1$. If all the states of $T$ induce the same permutation $\phi$ on the set $X_{n}$, then the minimal transducer representing $T$ has one state, and that state also induces the permutation $\rho$ on $X_{n}$. Therefore, suppose there are two states $p, q \in Q_{T}$ and $x \in X_n$ such that $t:=\lambda_{T}(x,p) \ne \lambda_{T}(x, q)=:z$. Since $\pi_{T}(x,p) = \pi_{T}(x,q)$, it follows that in $T^{-1}$, the state $p^{-1}, q^{-1}$ satisfy, $\pi_{T^{-1}}(t,p^{-1}) = \pi_{T^{-1}}(z, q^{-1})$. This yields the desired contradiction by Lemma~\ref{lem:letterspartitionstates}, since $B$ is isomorphic to $G(n,1)$.
	
	Now suppose the conclusion of the proposition holds for all transducers $T$ with $n \le |T| < m$ and which satisfy the hypothesis of the proposition.
	
	Let $T$ be a transducer with size $|T|=m$ satisfying the hypothesis. Let $(j,k)$ be the minimal bi-synchronizing level of $T$. Since $|T| >n$, it follows that both $j$ and $k$ are strictly greater than $1$. As, because $T$ is core, if $j$ or $k$ were $1$, $T$ or $T^{-1}$ would be a folding of $G(n,1)$ and so, $T$ and $T^{-1}$ would have size less than or equal to $n$. 
	
	Let $A$ and $B$ the automata obtained from $T$ and $T^{-1}$ respectively by forgetting outputs and let $(A_i)_{i \in \mathbb{N}}$ and $(B_i)_{i \in \mathbb{N}}$ be their respective  synchronizing sequences.

  Let $p,q \in Q_{T}$ be any pair of states satisfying $\pi_{T}(x,p) = \pi_{T}(x,q)$ for all $x \in X_{n}$. Then, by the argument given in the base case, we must also have $\lambda_{T}(x,p) = \lambda_{T}(x,q)$ for all $x \in X_{n}$, otherwise we obtain the contradiction that $T$ does not satisfy the hypothesis of the proposition. By the same argument, if $p^{-1},q^{-1} \in Q_{T^{-1}}$ are any pair of states satisfying $\pi_{T^{-1}}(x,p^{-1}) = \pi_{T^{-1}}(x,q^{-1})$ for all $x \in X_{n}$, then $\lambda_{T^{-1}}(x,p^{-1}) = \lambda_{T^{-1}}(x,q^{-1})$ for all $x \in X_{n}$ as well.
  
  Let $\sim$ be the equivalence relation on the states of $T$ given by $p \sim q$ if $\pi_{T}(x,p) = \pi_{T}(x,q)$ for all $x \in X_{n}$. By an abuse of notation we also use $\sim$ for the same equivalence relation on the states of $T^{-1}$. For $q \in Q_T$, let $\mathsf{q}$ be its equivalence class and let $\mathsf{Q}_{T}:= \{ \mathsf{q} \mid q \in Q_{T} \}$. Notice that, by the preceding paragraph, for states $p,q \in Q_T$, $p \sim q$ if and only if $\pi_{T}(\centerdot, p) = \pi_{T}(\centerdot, q)$ and $\lambda_{T}(\centerdot, p) = \lambda_{T}(\centerdot, q)$  if and only if $p^{-1}\sim q^{-1}$. Moreover, by  hypothesis, $\sim$ is not the trivial equivalence relation i.e its equivalence classes do not all consist of singleton sets and it also does not consist of one equivalence class.
  
   Form a new transducer $\mathsf{T}$ as follows. Let $Q_{\mathsf{T}}:= \mathsf{Q}_{T}$. Define the transition function $\pi_{\mathsf{T}}: X_n \times Q_{\mathsf{T}}\to Q_{\mathsf{T}}$ by $\pi_{\mathsf{T}}(x, \mathsf{q}) = \mathsf{p}$ where $p = \pi_{T}(x, q)$ for some $q \in \mathsf{q}$. The output function $\lambda_{\mathsf{T}}: X_n \times Q_{\mathsf{T}} \to X_n$ is defined by $\lambda_{\mathsf{T}}(x, \mathsf{q}) = \lambda_{T}(x,q)$ for some $q \in \mathsf{q}$. The preceding paragraph implies that $\mathsf{T}$ is well-defined. 
   
   Observe, that if $C$ is the automaton obtained from $\mathsf{T}$ by forgetting outputs and $D$ is the automaton obtained from $\mathsf{T}^{-1}$ by forgetting outputs, then $C$ is isomorphic to $A_1$ and $D$ is isomorphic to $B_1$, by definition of the synchronizing sequence. This means that the minimal bi-synchronizing level of $\mathsf{T}$ is $(j-1, k-1)$. Moreover, as $k-1$ and $j-1$ are at least $1$, in the synchronizing sequence of $C$ and $D$,  the terms $C_{k-2}$ and $D_{k-2}$ are isomorphic to $G(n,1)$. This means that  $\mathsf{T}$ satisfies the hypothesis of the proposition. Furthermore, as $\sim$ is not the trivial relation, we have $|\mathsf{T}|< |T|$. Thus, we conclude that the minimal transducer representing $\mathsf{T}$ has only one state and $j-1 = k-1$. However, by construction of $\mathsf{T}$, the minimal transducer representing $\mathsf{T}$ is also the minimal transducer representing $T$. This concludes the proof. \qed
\end{proof}

We have some corollaries of the result above.

\begin{cor}\label{cor:autfold}
	Let $A$ be a folded automaton arising from a folding of $G(n,m)$. If an element of the synchronizing sequence of $A$ is isomorphic to $G(n,1)$, then any automorphism of $G_{A}$ is induced by a permutation of the symbol set $X_{n}$.
\end{cor}
\begin{proof}
	Let $\phi$ be any automorphism of $G_{A}$, and let $H:= H(A, \phi)$. Let $A(H)$ and $A(H^{-1})$ be the automata obtained from $H$ and $H^{-1}$ by forgetting outputs. Note that since $A(H)$ and $A(H^{-1})$ are isomorphic as automata to $A$, it follows that $H$ satisfies the hypothesis of Proposition~\ref{prop:onestatecond}. This means that the minimal representative of $H$ has exactly one state. Proposition~\ref{prop:permaut} now implies that $\phi$ is  induced by a permutation of the symbol set $X_{n}$.\qed
\end{proof}

\begin{cor}\label{cor:autdebruijn}
	Let $A$ be the de Bruijn automaton $G(n,m)$. Then $\aut{G_{A}}$ is isomorphic to the symmetric group on $n$ points.
\end{cor}
\begin{proof}
	$G(n,m)$ is clearly a folding of itself, thus Corollary~\ref{cor:autfold} implies that the automorphism group of its underlying directed graph is isomorphic to a subgroup of the symmetric group on $n$ points. However, we have seen above that any permutation of $X_n$ induces an automorphism of $G(n,m)$.\qed
\end{proof}

The corollaries of Proposition~\ref{prop:onestatecond} below require the following straight-forward lemma.

\begin{lemma}\label{lem:2letter}
	Let $A$ be any strongly synchronizing, core automaton over the $2$-letter alphabet. Let $(A_{i})_{i \in \mathbb{N}}$ be the synchronizing sequence of $A$. if $|A|>1$, then the minimal synchronizing level $k$ of $A$ is at least $1$ and $A_{k-1}$ is isomorphic to $G(2,1)$. 
\end{lemma}
\begin{proof}
	If $|A|>1$ then it is not the single state automaton (which is the only automaton strongly synchronizing at level 0). Thus let $k \ge 1$ be the minimal synchronizing level of $A$. Now, since $A$ is core, it follows that the automaton $A_{k-1}$ is isomorphic to $G(2,1)$. This is because the only core, level 1 synchronizing automaton over the 2 letter alphabet is $G(2,1)$.\qed
\end{proof}

\begin{cor}
	Let $A$ be an folded automaton over the $2$ letter alphabet, then $\aut{G_A}$ is either trivial or the cyclic group of order $2$. Moreover any automorphism of $G_{A}$ is induced by a permutation of $X_{2}$.
\end{cor}
\begin{proof}
	This is a direct consequence of Lemma~\ref{lem:2letter} and Corollary~\ref{cor:autfold}.\qed
\end{proof}

\begin{theorem}\label{cor:H2isC2}
	The group $\hn{2}$ is isomorphic to the cyclic group of order $2$.
\end{theorem}
\begin{proof}
	Let $A$ be a minimal, core, bi-synchronizing transducer over the $2$ letter alphabet. Suppose for a contradiction that $|A| >1$. By Lemma~\ref{lem:2letter}, $A$ satisfies the hypothesis of Proposition~\ref{prop:onestatecond}. However, this yields a contradiction as the size of $A$ must then be $1$. 
	
	Thus, every element of $\hn{2}$ has exactly one state yielding the result. \qed
\end{proof}

\section{Decomposing elements of $\hn{n}$}\label{sec:decomposition}

In this section we give an algorithm for decomposing an arbitrary element of $\hn{n}$ as a product of elements arising from automorphisms of the directed graphs underlying the folded automata arising from foldings of $G(n,m)$. Our method can be thought of as an interpretation of the approach in  \cite{BoyleFranksKitchens} in the language of strongly synchronizing automata. However, we are able to simplify that approach a great deal. In particular, we show that an element $T \in \hn{n}$ of size $l$ for some $l \in \N$ can be written as a product of at most $l$ elements of $\hn{n}$ arising from automorphisms of directed graphs underlying foldings of the underlying automaton of $A^{-1}$. Note that an element $T \in \hn{n}$ of size $l$ is strongly synchronizing at level at most $l-1$, thus $f_{T}$ (by Remark~\ref{bijectionfromFinftytoPn}) corresponds to a map $f_{\infty}$ for some $f \in F(\xn, l)$. To decompose the element $f_{\infty}$ using the approach given in \cite{BoyleFranksKitchens}, one would first have to construct a graph (isomorphic to the underlying graph of some folded automaton) with at least $n^{l}$ vertices.

\subsection{Collapsing chains and amalgamation}
 We introduce some terminology. Let  $A$ and $B$ be automata. Then $B$ is said to belong to  \emph{a collapse chain} of $A$, if there is a sequence $$A = A_0, A_1, \ldots, A_{m} = B$$ where, for $i \ge 1$, $A_{i}$ is obtained from $A_{i-1}$  by identifying two states $p,q \in Q_{A_{i-1}}$ such that $\pi_{A_{i-1}}(\cdot, p) =\pi_{A_{i-1}}(\cdot, q)$. We stress  that each  term in the sequence  is obtained from the previous one by identifying exactly two states. Observe that if $A$ is strongly synchronizing and $B$ is an automaton which is in a collapse chain of $A$, then $B$ is synchronizing at the minimal synchronizing level of $A$. More generally, let $A$ and $B$ be automaton with $B$ in a collapse chain of $A$. Let $(A_i)_{i \in \N}$ and $(B_i)_{i \in \N}$ be the sychronizing sequences of $A$ and $B$ respectively, and suppose that $k,l \in \N$ are minimal such that $A_j = A_{k} $ for all $j \ge k$ and $B_{j} = B_{l}$ for all $j \ge l$, then $A_{k} = B_{l}$. This is a consequence of Theorem~\ref{thm:collapsingprocedure}.
 
 The following terminology, which is for the underlying graphs of an automaton, should be compared with the similarly named terminology in the paper \cite{BoyleFranksKitchens} (recall the direction of edges will be reversed in our context). Let $A$ and $B$ be automata. Then $G_{B}$ is called an \emph{amalgamation} of $G_{A}$ if there is a sequence $G_A= G_0, G_1, \ldots, G_m = G_{B}$ where, for $i \ge 1$, $G_{i}$ is obtained from $G_{i-1}$ by identifying two vertices $v_1$ and $v_2$ of $G_{i-1}$ having the property that for all vertices $v$ of $G_{i-1}$, if there are precisely $k$ outgoing edges from $v_1$ to $v$ (for some $1 \le k \le n$) then there are also precisely $k$ outgoing edges from $v_2$ to $v$. That is, we replace the vertices $v_1$ and $v_2$ with a single vertex $v_{1,2}$ and, for every vertex $v$ of $G_{i-1}$ if there are $k$ edges from $v_1$ to $v$ (and hence, from $v_2$ to $v$), then there are $k$ edges from $V_{1,2}$ to $v$ (and of course, we retain all other vertices and edges of $G_{i-1}$).  Also, if $v$ is a vertex of $G_{i-1}$ then there will be $t$ edges in $G_{i}$ from the vertex corresponding to $v$ to $v_{1,2}$ if the cardinality of the set of edges from $v$ to $v_1$ is $r$ while the cardinality of the set of edges from $v$ to $v_2$ is $s$, where $r+s=t$.  In particular, if there are $m$ loops based as $v_1$ and $m'$ loops based at  $v_2$ in $G_{i-1}$, there are exactly $m+m'$ loops based at $v_{1,2}$ in $G_{i}$.  In this context, the vertices $v_1$ and $v_2$ of $G_{i-1}$ are called \emph{amalgamable}.

  Let $T$ be an invertible transducer. Let $A$ and $B$ be the underlying automata of $T$ and $T^{-1}$ respectively. Let $(B_i)_{i \in \N}$ be the synchronizing sequence of $B$. Then, by definition of the inverse transducer, $G_{B_i}$ is an amalgamation of $G_{A}$ for all $i \in \N$. The condition  ``for two states $p^{-1}, q^{-1} \in Q_{T^{-1}}$, $\pi_{B}(\cdot, p^{-1})$ and $\pi_{B}(\cdot, q^{-1})$ are equal'' is equivalent to the condition ``the vertices $p$ and $q$ of $G_{A}$ are amalgamable''.  Further observe that for automata $A$ and $B$ with $B$ in a collapse chain of $A$, the underlying directed graph of one is an amalgamation of the other.

  \subsection{Description of the decomposition algorithm}
  Here we give a short description of the algorithm for decomposing an element $T$ of $\hn{n}$ as a product of torsion elements as described in Theorem \ref{Thm:decompositionIntro}.  The proof that our various steps can be carried out is given in full detail in Subsection \ref{ssec:decompProof}.  The algorithm allows the user some choices, so decomposition is not unique, but our upper bound on the decomposition length still holds.
  
  We conclude with an example decomposition  and statements of choices we made so the reader can verify by following the algorithm.
  
  \begin{enumerate}[label= \textbf{A\arabic*}]
  \item Let $T_0\in \hn{n}$. Let $A$ and $B$ be the underlying automata of $T_0$ and $T_0^{-1}$ respectively.
  \item If $T_0$ has only one state, then it represents a permutation, and so there is a finite order single state transducer that we can multiply against $T_0$ to produce the identity element (in this case, go to the final step of the algorithm with this finite order factor in hand).  Otherwise, proceed to the next step.
  \item Compute the synchronizing sequence $(B_i)_{i \in \N}$ for $B=B_0$.
  \item Compute the first step $A_1$ of the synchronizing sequence of $A=A_0$.
  \item Find a pair $(p,q)$ of distinct states of $A$ which belong to the same state of $A_1$.
  \item Find the non-identity permutation $\alpha$ of the output labels such that $\lambda(\cdot, q) \circ \alpha: \xn \to \xn$ is precisely $\lambda(\cdot, p): \xn \to \xn$. Determine the disjoint cycle decomposition of the permutation $\alpha$.
  \item There is a smallest index $i$ so that the state $[q]$ of the automaton $B_i$ has the following properties:
  \begin{itemize}
      \item The states $[q]$ and $[p]$ remain distinct states of $B_{i}$, and
      \item For all $x,y \in \xn$ belonging to the same disjoint cycle in the cycle decomposition of $\alpha$, the edges labelled $x$ and $y$ from $[q]$ are parallel edges.
    \end{itemize}
 Now determine the isomorphism $\tau_{\alpha}$ of $B_i$ which fixes all vertices and induces the permutation $\alpha$ on the edges leaving $[q]$.)
  \item Build the transducer $H(B_i,\tau_{\alpha})$.  This is a  factor of finite order in a product sequence that will eventually trivialize $T_0$.
  \item Compute the product $R=\core(T \ast H(B_i,\tau_{\alpha}))$.  This product has the same underlying graph as $T$ but is not minimal.  The states corresponding to $p$ and $q$ in this product are $\omega$-equivalent, and will be identified by minimizing the result $R$ to produce a new element $T_1$ with fewer states than $T_0$.  
  \item Repeat this process from the beginning, remembering the list of finite factors found so far.
  \item The transducer $T$ now factors as the product in reverse order of the inverses of the finite order factors found above.
  \end{enumerate}
  
We give an example. Consider the element $T:=H(A, \phi)$ from Figure~\ref{fig-transducernonpermaut}. Working through the  algorithm, with $p = q_1,$ and $q = q_0$ in the first instance, one obtains the following decomposition below (up to changing the final single state transducer; different choices for $p$ and $q$ in building the second factor result in different single-state third-factor transducers):

\begin{figure}[H]
	\begin{center}
		\begin{tikzpicture}[shorten >= .5pt,node distance=3cm,on grid,auto] 
		\node[xshift=-1.5cm, yshift=-1cm] (T) {$T = $};
		\node[state, xshift=6cm] (q0)   {$q_0$};
		\node[state, xshift= 6cm, yshift=-2cm] (q1)   {$q_1$};
		\node[state, xshift=3cm, yshift=0cm ] (p0) {$p_0$};
		\node[state, xshift=3cm,yshift=-2cm] (p1) {$p_1$};
		\node[state, xshift= 0cm, yshift=-1cm] (t)   {$t$};
		\node[xshift= 1.5cm, yshift=-1cm] {$\spnprod{n}$};
		\node[xshift= 4.5cm, yshift=-1cm] {$\spnprod{n}$};
		\path[->] 
		
		(t) edge[out=40, in=140,loop]  node[swap]{$0|2$, $1|1$, $2|0$} ()

		(p0) edge[out=75, in=105,loop]  node[swap] {$2|2$, $1|1$} ()
		edge[out=260, in=100] node[swap] {$0|0$} (p1) 
		
		(p1) edge[in=255, out=285,loop] node{$0|0$}   ()
		     edge[in=280, out=80] node[swap, yshift=-0.3cm] {$2|1$} node[swap,xshift=0cm, yshift=0.2cm]{$1|2$}  (p0)
		
		(q0) edge[out=75, in=105,loop] node[swap]{$0|2$} ()
		edge[out=280, in=80] node[yshift=-0.3cm] {$0|1$} node[xshift=0cm, yshift=0.2cm]{$1|0$} (q1)
		
		(q1) edge[in=255, out=285,loop]  node{$1|1$, $0|0$} ()
		     edge[out=100, in=260] node {$2|2$} (q0);
		\end{tikzpicture}
	\end{center}
	\caption{Decomposing an element of $\hn{3}$ as a product of involutions.\label{fig-decomposition}}
\end{figure}
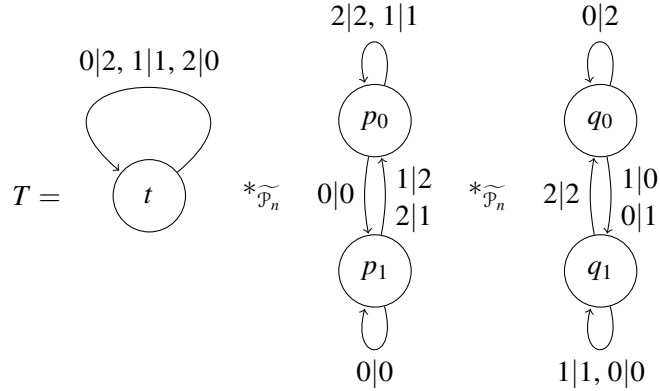

\subsection{Proof of Theorem \ref{Thm:decompositionIntro}}
\label{ssec:decompProof}
Here we prove that the algorithm above works.

  Recall that $\shn{n}$ consist of those transducers which are strongly synchronizing and have an automata-theoretic inverse but which do not necessarily induce homeomorphisms of $\xnz$. Further recall that for $T, U \in \spn{n}$ the product, in $\pn{n}$, of $T$ and $U$ is obtained by identifying the $\omega$-equivalent states of $\core(T \ast U)$, write $T \spnprod{n} U$ for this transducer.
 
 \begin{lemma}\label{lemma:decompnogrowth}
    Let $A$ be a minimal transducer in $\hn{n}$. Let $B$ be the underlying automaton of $A^{-1}$ and $(B_{i})_{i \in \N}$ be the synchronizing sequence of $B$. Let $H \in \shn{n}$ be any transducer such that the underlying automaton of $H$ is $B_{j}$ for some  $j \in \N$. For a state $p^{-1}$ of $A^{-1}$ write $[p^{-1}]$ for the state of $B_{j}$ containing $p^{-1}$. Then 
    \begin{enumerate}
        \item the set of states of $\core(A \ast H)$ is precisely the set $\{(p, [p^{-1}]) \mid p \in Q_{A} \}$. Consequently,
        \item $|A \spnprod{n} H| \le |A|$, and,
        \item the underlying automaton of $A \spnprod{n} H$ belongs to a collapse chain of the underlying automaton of $A$.
    \end{enumerate}
    
 \end{lemma}
 \begin{proof}

     Let $p \in Q_{A}$ and $x \in \xn$ and consider the transition $\pi_{A\ast H}(x,(p, [p^{-1}]))$. Let $y = \lambda_{A}(x, p)$ and $q = \pi_{A}(x, p)$. Then, in $A^{-1}$, we have $\pi_{A^{-1}}(y, p^{-1}) = q^{-1}$, therefore, in $B_{j}$, $\pi_{B_{j}}(y, [p^{-1}]) = [q^{-1}]$. Hence, we conclude that $$\pi_{A\ast H}(x,(p, [p^{-1}])) = (q, [q^{-1}]).$$ 
     
     To see that  $(p, [p^{-1}])$ is a state of $\core(A \ast H)$, let $\gamma \in \xnp$ be a word such that $\pi_{A}(\gamma, p) = p$. The preceding paragraph now shows that $\pi_{A \ast H}(\gamma, (p, [p^{-1}])) = (p, [p^{-1}]) $.
     
     Thus we see that $|\core(A \ast H)| = |A|$. In particular the underlying automaton of $\core(A \ast H)$ is isomorphic as an automaton to the underlying automaton of $A$ via the map sending $(p, [p^{-1}])$ to $p$.
     
     Now, $A \spnprod{n} H$ is obtained by  identifying $\omega$-equivalent states of $\core(A \ast H)$. Therefore the underlying automaton of $A \spnprod{n} H$ belongs to a collapse chain  of the underlying automaton of  $\core(A \ast H)$ as required.\qed
     \end{proof}
     
     \begin{lemma}\label{Lemma:decompfindingautomorphism}
         Let $A \in \hn{n}$ be a minimal transducer, let $B$ be the underlying automaton of $A^{-1}$ and $(B_{i})_{i \in \N}$ be the synchronizing sequence of $B$. Suppose there are distinct states $q_1, q_2 \in Q_{A}$  such that the maps $\pi_{A}(\cdot, q_1)$ and $\pi_{A}(\cdot, q_2)$ are equal. Then there is a transducer $H$ with the following properties:
         \begin{enumerate}
             \item there is a $j \in \N$ such that $H = H(B_{j}, \phi)$ for an automorphism $\phi$ of $B_{j}$ and,
             \item writing $[q^{-1}]$ for the state of $B_{j}$ containing $q^{-1}$, $q \in Q_{A}$, we have $$\lambda_{A}(\cdot, q_2) \circ \lambda_{H(B_{j}, \phi)}(\cdot, [q_2^{-1}]): \xn \to \xn$$ is precisely the map $\lambda_{A}(\cdot, q_1): \xn \to \xn$. 
         \end{enumerate}
     \end{lemma}
     \begin{proof}
         Since $q_1, q_2$ are distinct states of $A$ and since $A$ is minimal, the states $q_1$ and $q_2$ are not $\omega$-equivalent. Therefore, there are $x\neq y$ and $z \in \xn$ such that $\lambda_{A}(x,q_1) = \lambda_{A}(y,q_2) = z$. Let $p_1 = \pi_{A}(x,q_1)$ and $p_2 = \pi_{A}(y,q_2)$. In $A^{-1}$, we have $\pi_{A^{-1}}(z, q_1^{-1}) = p_{1}^{-1}$ and $\pi_{A^{-1}}(z, q_{2}^{-1}) = p_2^{-1}$. Since $A^{-1}$ has minimal synchronizing level $k$, it therefore follows that either $k=1$ and $p_1 = p_2$ or $k \ge 2$ and the maps $\pi_{A^{-1}}(\cdot, p_1^{-1}): \xn^{k-1} \to Q_{A^{-1}}$ and $\pi_{A^{-1}}(\cdot, p_2^{-1}): \xn^{k-1} \to Q_{A^{-1}}$ are equal. Therefore, by Theorem~\ref{thm:collapsingprocedure}, the minimal $j \in \N$ for which $p_1^{-1}$ and $p_2^{-1}$ belong to the same state of $B_{j}$ is at most $k-1$.
	
	Define a relation $\mathscr{R}$ on the set of states $$Q_{q_1^{-1},q_2^{-1}}:= \{ p^{-1} \in Q_{A^{-1}} \mid \exists x \in \xn, a \in \{1,2\}: \pi_{A^{-1}}(x,q_a^{-1}) = p^{-1} \}$$ by setting $p^{-1} \mathscr{R} q^{-1}$ if and only if  there is a letter $z \in \xn$ such that $$\pi_{A^{-1}}(z,q_1^{-1}) = p^{-1} \mbox{ and } \pi_{A^{-1}}(z,q_2^{-1}) = q^{-1}.$$ Let $\overline{\mathscr{R}}$ be the transitive closure of $\mathscr{R}$, so that $\overline{\mathscr{R}}$ is an equivalence relation on $Q_{q_1^{-1},q_2^{-1}}$. By the preceding paragraph, for a state $p^{-1} \in Q_{q_1^{-1},q_2^{-1}}$, there is a minimal $j \in \N$, $j \le k-1$, and all elements of $[p^{-1}]_{\overline{\mathscr{R}}}$, the equivalence class of $p^{-1}$, belong to the same state of $B_{j}$. 
	
	Let $J \in \N$, $J \le k-1$, be minimal such that for any $p^{-1} \in Q_{A^{-1}}$ there is a state of $B_{J}$ such that all elements of $[p^{-1}]$ belong to the same state of $B_{J}$. Observe that if $\mathscr{R}$ is the diagonal relation, that is, if $\mathscr{R}$ is precisely the set $\{(p^{-1}, p^{-1}) \mid p^{-1} \in Q_{q_1^{-1},q_2^{-1}}\}$, then $B_{j} = B_{0}$. Further observe that $\mathscr{R}$ is the diagonal relation precisely when for all $x \in \xn$, $\pi_{A^{-1}}(x, q_1^{-1}) = \pi_{A^{-1}}(x, q_2^{-1})$. If there is $z \in \xn$, such that $\pi_{A^{-1}}(z, q_1^{-1}) \ne \pi_{A^{-1}}(z, q_2^{-1})$, then minimality of $J$ forces that the states $q_1^{-1}$ and $q_{2}^{-1}$ do not belong to the same state of $B_{J}$. Therefore,  as $q_1$ and $q_2$ are distinct states of $A$,  they are contained in distinct states of $B_{J}$.
	
	Let $t_1$ and $t_2$ be the distinct states of $B_{J}$ containing $q_1^{-1}$ and $q_{2}^{-1}$ respectively. Observe that the maps $\pi_{B_{J}}(\cdot, t_1)$ and $\pi_{B_{J}}(\cdot, t_2)$ are equal by choice of $J$ and definition of the relation $\overline{\mathscr{R}}$. Define a map  $\lambda_{B_{J}}(\cdot, t_2): \xn \to \xn$ as follows. Let $x, \in  \xn$ and let $z = \lambda_{A}(x, q_1)$ and $y = \lambda_{A}(x, q_2)$ then set $\lambda_{B_{J}}(y, t_2): = z$. Since $\lambda_{A}(\cdot, q_1)$ and $\lambda_{A}(\cdot, q_2)$ are permutations of $\xn$, then $\lambda_{A}(\cdot, t_2)$ is a bijection as well. Moreover we note that $$\lambda_{A}(\cdot, q_2)\circ \lambda_{A}(\cdot, t_2): \xn \to \xn$$ is precisely the map $\lambda_{A}(\cdot, q_1): \xn \to \xn$.
	
	Let $a,b,c \in \xn$ be arbitrary such that $\lambda_{B_{J}}(a, t_2)= b$ and $\lambda_{B_{J}}(b, t_2) = c$. By definition, there are $x,y \in \xn$ such that $\lambda_{A}(x, q_1) = a$, $\lambda_{A}(x, q_2) = b$, $\lambda_{A}(y, q_1) = b$ and $\lambda_{A}(y, q_2) = c$. By the assumption that $\pi_{A}(\cdot, q_1)$ and $\pi_{A}(\cdot, q_2)$ are equal, we have $p^{-1} := \pi_{A^{-1}}(a, q_1^{-1}) =\pi_{A^{-1}}(b, q_2^{-1})$ and $q^{-1}:= \pi_{A^{-1}}(b, q_1^{-1}) = \pi_{A^{-1}}(c,q_2^{-1})$. Thus, $p^{-1}$  is $\mathscr{R}$ related to $q^{-1}$.  Therefore, it is the case that $\pi_{B_{J}}(b, t_2) = \pi_{B_{J}}(c, t_2)$.
	
	Let $(x_1 \ x_2 \ x_3 \ \ldots \ x_m)$ be a sequence of elements of $\xn$ such that for $1 \le i \le m-1$, $\lambda_{B_{J}}(x_i, t_2) = x_{i+1}$ and $\lambda_{B_{J}}(x_{m}, t_2) = x_1$. By an induction argument making use of the previous paragraph we see that there is a state $t \in Q_{B_{J}}$ such that $\pi_{B_{J}}(x_i, t_2) = t$ for all $1 \le i \le m$. Thus, it follows that given $a,b \in \xn$ such that $\lambda_{B_{J}}(a, t_2)= b$ then, $\pi_{B_{J}}(a, t_2) = \pi_{B_{J}}(b, t_2)$.
	
	Let $t$ be any state of $B_{J}$ not equal to $t_2$, we set $\lambda_{B_J}(\cdot, t) : \xn \to \xn$ to be the identity permutation. Set $H(B_j):= (\xn, Q_{B_{J}}, \pi_{B_J}, \lambda_{B_J})$.  
	
	Let $\phi$ be the automorphism of $G_{B_J}$ which fixes all vertices of $G_{B_j}$  and whose action on the edges of $G_{B_{j}}$ is as follows. For an edge $(t_2, x, t)$ of $G_{B_{J}}$ with initial vertex $t_2$, set $(t_2, x, t)\phi := (t_2, \lambda_{B_J}(x,t_2), t)$; $\phi$ fixes every other edge. It is clear from the preceding paragraphs that $H(B_{J}, \phi) = H(B_{J})$. Thus we may take $H = H(B_{J})$ concluding the proof. \qed
     \end{proof}
 
\begin{prop}\label{prop:decomposition}
	Let $A \in \hn{n}$ be a minimal transducer, $B$ be the underlying automaton of $A^{-1}$, $(B_{i})_{i \in \N}$ be the synchronizing sequence of $B$ and $k \in \N$ be minimal such that $|B_{k}| =1$. Suppose there are distinct states $q_1, q_2 \in Q_{A}$  such that the maps $\pi_{A}(\cdot, q_1)$ and $\pi_{A}(\cdot, q_2)$ are equal.
	Then, there is an $i \in \N$, and an automorphism $\phi$ of $G_{B_{i}}$ fixing vertices and such that  $|A\spnprod{n}H(B_i, \phi)|< |A|$. Thus, $G_{H(B_{i},\phi)} = G_{B_i}$ is an amalgamation of $G_{A}$. Moreover, the underlying automaton of $A\spnprod{n}H(B_i, \phi)$ is belongs to a collapse chain of the underlying automaton of $A$. Therefore, $(A\spnprod{n}H(B_i, \phi))$   has minimal synchronizing  level at most the minimal synchronizing level of $A$.
\end{prop}
\begin{proof}

By Lemma~\ref{Lemma:decompfindingautomorphism} there is a transducer $H$ with the following properties:
         \begin{itemize}
             \item there is a $j \in \N$ such that $H = H(B_{j}, \phi)$ for an automorphism $\phi$ of $B_{j}$ and,
             \item writing $[q^{-1}]$ for the state of $B_{j}$ containing $q^{-1}$, $q \in Q_{A}$, we have $$\lambda_{A}(\cdot, q_2) \circ \lambda_{H(B_{j}, \phi)}(\cdot, [q_2^{-1}]): \xn \to \xn$$ is precisely the map $\lambda_{A}(\cdot, q_1): \xn \to \xn$. 
         \end{itemize}

The result follows by applying Lemma~\ref{lemma:decompnogrowth} to the product  $A \spnprod{n} H$. \qed       
	
\end{proof}

\begin{theorem}\label{Thm:decomposition}
	Let $T \in \hn{n}$, $A$ the underlying automaton of $T$, $(A_{i})_{i \in \N}$ the synchronizing sequence of $A$ and  $k$ be minimal such that $A_{j} = A_{k}$ for all $j \ge k$. Note that since $T$ is strongly synchronizing, $A_{k} = 1$. Then $T$ can be written as a product of a single state transducer $U$ and at most $|A|-1$ elements of $\hn{n}$ which arise from vertex-fixing automorphisms  of directed graphs  which are amalgamations of $G_{A}$. 
\end{theorem}
\begin{proof}
	The proof follows by repeatedly applying Proposition~\ref{prop:decomposition}.\qed
\end{proof}

We note that Theorem~\ref{Thm:decompositionIntro} is a corollary of Theorem~\ref{Thm:decomposition} above.

\begin{lemma}\label{lem:mutltiedges}
	Let $A$ be a strongly synchronizing core automaton with more that one state. Then for any pair $p,q \in Q_{A}$ there are is a least one element $x \in \xn$ such that $\pi_{A}(x,p) \ne q$. In other words, there are at most $n-1$ edges in $G_{A}$ from the vertex $p$ to the vertex $q$.
\end{lemma}
\begin{proof}
	Suppose for a contradiction that there are states $p,q \in Q_{A}$ such that $\pi_{A}(x,p) = q$ for all $x \in \xn$. Let $(A_{i})_{i \in \N}$ be the synchronizing sequence of $A$, and let $k$ be minimal such that $A_{k} = 1$. Notice that $A_{k-1}$ is synchronizing at level $1$ and core and has more than one state by assumption on $k$. Let  $t$ be the state of $A_{k-1}$ which contains $p$, and $t'$ be the state of $A_{k-1}$ containing $q$. It follows that $\pi_{A}(x_{n},t) = t'$ for all $x \in \xn$. Since $A_{k-1}$ is synchronizing at level $1$, this forces, $|A_{k-1}| = 1$ which yields the desired contradiction. \qed
\end{proof}

\begin{cor}\label{cor:3letter}
	Let $A$ be a strongly synchronizing core automaton over the alphabet $X_{3}$ with more than one state. Let $\phi$ be any automorphism of $G_{A}$ that fixes vertices, then $\phi$ has order at most $2$.
\end{cor}

\begin{cor}\label{cor:decomposition3}
	Let $T \in \hn{3}$, $A$ be the underlying automaton of $T$ and $(A_{i})_{i \in \N}$ be the synchronizing sequence of $A$. Let  $k \in \N$ be minimal such that $|A_{k}|=1$. Then $T$ can be written as a product of a  single state transducer $U$ and at most $|A|-1$ elements of $\hn{n}$ of order $2$ which arise from vertex-fixing automorphisms  of directed graphs  which are amalgamations of $G_{A}$. 
\end{cor}
\begin{proof}
	The proof follows by repeated applications of Proposition~\ref{prop:decomposition} and Corollary~\ref{cor:3letter}. \qed
\end{proof}

We generalize Corollary~\ref{cor:decomposition3} to all $n$. However, the number of elements of order $2$ required is bigger than the number of states  in general. We require first the following straight-forward observation.

\begin{lemma}
	Let $G$ be a directed graph and $\phi$ be an automorphism of $G$ that fixes vertices. Then $\phi$ can be written as a product of vertex-fixing automorphisms of $G$ of order $2$.
\end{lemma}

\begin{cor}\label{cor:decompositioninvolution}
	Let $T \in \hn{n}$, $(A_{i})_{i \in \N}$ be the synchronizing sequence of $A$ and  $k$ be minimal such that $A_{j} = A_{k}$ for all $j \ge k$. Then $T$ can be written as a product of a single state transducer $U$ with underlying automaton $A_{k}$, and elements of $\hn{n}$ of order $2$ arising from vertex-fixing automorphisms  of directed graphs  which are amalgamations of $G_{A}$. 
\end{cor}

It is possible to bound the number of involutions appearing in Corollary~\ref{cor:decompositioninvolution} in terms of $A$ (i.e the number of vertices and edges of $G_{A}$) but we have not attempted to do so.

\section{Counting foldings}\label{sec:counting}

Counting foldings of the de Bruijn graph $G(n,k)$ is an important and
challenging problem. We give here the solution for $k=1$ (which is trivial)
and for $k=2$.

The \emph{Bell number} $B(n)$ is the number of partitions of an $n$-set.
This well-studied combinatorial sequence is given by the recurrence relation
\[B(n)=\sum_{k=1}^n{n-1\choose k-1}B(n-k)\]
for $n>0$, with $B(0)=1$.

\begin{prop}
	The number of foldings of $G(n,1)$ is the Bell number $B(n)$.
\end{prop}

\begin{proof}
	The vertex set is identified with $X_{n}$, so any folding is a partition of $X_n$;
	and clearly any partition of $X_n$ is a folding.\qed
\end{proof}

\begin{theorem}
	The number of foldings of the de Bruijn graph with word length $2$ over an
	alphabet of cardinality $n$ is
	\[\sum_\pi\prod_{i=1}^{|\pi|}R(|\pi|,|A_i|),\]
	where $\pi$ runs over partitions of the alphabet, $A_i$ is the $i$th part, and
	\[R(s,t)=\sum_\rho(-1)^{|\rho|-1}(|\rho|-1)!\prod_{i=1}^{|\rho|}B(|C_i|s),
	\]
	where $\rho$ runs over all partitions of $\{1,\ldots,t\}$, and $C_i$ is the
	$i$th part.
\end{theorem}

The formula is somewhat complicated, but values are easily computed (and
rapidly growing): the numbers for $n=1,\ldots,7$ are
$1$, $5$, $192$, $78721$, $519338423$, $82833228599906$,
$429768478195109381814$.

\begin{proof}
	
	We define a graph $\Gamma$ associated with a folding: the vertex set is the
	alphabet $X_n$, and two vertices $x$ and $y$ are joined if there exist $u$ and
	$v$ such that $ux\equiv vy$.
	
	\begin{center}
		\setlength{\unitlength}{1mm}
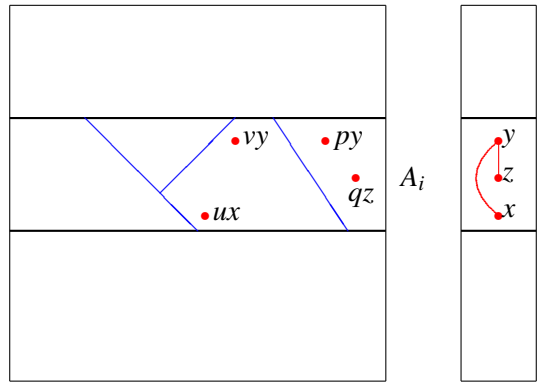
		\begin{picture}(70,50)
		\multiput(0,0)(50,0){2}{\line(0,1){50}}
		\multiput(0,0)(0,50){2}{\line(1,0){50}}
		\thicklines
		\multiput(0,20)(0,15){2}{\line(1,0){50}}
		\thinlines
		\color{blue}
		\put(10,35){\line(1,-1){15}}
		\put(20,25){\line(1,1){10}}
		\put(35,35){\line(2,-3){10}}
		\color{red}
		\put(26,22){\circle*{1}}
		\put(30,32){\circle*{1}}
		\put(42,32){\circle*{1}}
		\put(46,27){\circle*{1}}
		\color{black}
		\put(27,21.5){$ux$}
		\put(31,31.5){$vy$}
		\put(43,31.5){$py$}
		\put(45,24){$qz$}
		\put(52,26){$A_i$}
		\multiput(60,0)(10,0){2}{\line(0,1){50}}
		\multiput(60,0)(0,50){2}{\line(1,0){10}}
		\thicklines
		\multiput(60,20)(0,15){2}{\line(1,0){10}}
		\thinlines
		\color{red}
		\multiput(65,22)(0,5){3}{\circle*{1}}
		\put(65,27){\line(0,1){5}}
		\curve(65,22,62,27,65,32)
		\color{black}
		\put(65.5,21.8){$x$}
		\put(65.5,26.8){$z$}
		\put(65.5,31.8){$y$}
		\end{picture}
	\end{center}
	
	Let $\pi$ be the partition of $X_n$ into connected components of the graph
	$\Gamma$. If $A_i$ is a part of $\Gamma$, then the set $X_n\times A_i$ (the
	horizontal stripe in the figure) is a union of parts of the folding: no part
	can cross into a different horizontal stripe.
	
	Moreover, by the definition of a folding, we see that if $x,y\in A_i$, then
	$xw$ and $yw$ lie in the same part of the folding.
	
	\begin{center}
		\setlength{\unitlength}{1mm}
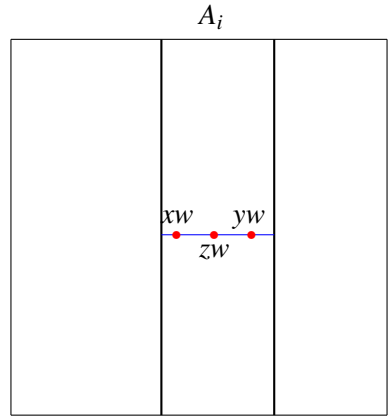
		\begin{picture}(50,55)
		\multiput(0,0)(50,0){2}{\line(0,1){50}}
		\multiput(0,0)(0,50){2}{\line(1,0){50}}
		\thicklines
		\multiput(20,0)(15,0){2}{\line(0,1){50}}
		\thinlines
		\color{blue}
		\put(20,24){\line(1,0){15}}
		\color{red}
		\multiput(22,24)(5,0){3}{\circle*{1}}
		\color{black}
		\put(20,25.5){$xw$} \put(25,21){$zw$} \put(29.5,25.5){$yw$}
		\put(25,52){$A_i$}
		\end{picture}
	\end{center}
	
	The sets $X_n\times A_i$ can be treated independently, so we have to count the
	number of good partitions of each and multiply them. Moreover, by the last
	remark, we can shrink each horizontal interval $A_j\times\{v\}$ to a point,
	so we have to partition $\pi\times A_i$.
	
	There are $B(|\pi|\cdot|A_i|)$ partitions of $\pi\times A_i$. We have to filter
	out the ones which do not induce partitions of $\pi\times B$ for any proper
	subset $B$ of $A_i$. By M\"obius inversion \cite[Section 3.7]{RStanley} over the lattice of partitions of
	$A_i$, we find that the number of these is $R(|\pi|,|A_i|)$, where $R$ is as
	defined earlier.
	
	Putting all this together gives the result.\qed
\end{proof}

Apart from this result, only a few values of the function counting foldings
are known: $G(2,3)$ has $30$ foldings, while $G(2,4)$ has
$1247$. (These numbers were obtained by brute-force computation.)

\section*{Acknowledgements}
The authors are grateful for partial support from EPSRC research grant EP/R032866/1. The third author is additionally grateful for support from Leverhulme Trust Research Project Grant RPG-2017-159 and for the warm hospitality of the University of Aberdeen where some of this research was conducted. We are also grateful to Mike Boyle, Matthew G. Brin,  Elliot Cawtheray, Timothy Gowers, and anonymous referees for comments on drafts of this article.

\bibliographystyle{amsplain}
\def\cprime{$'$} \def\cprime{$'$} \def\cprime{$'$} \def\cprime{$'$}
\providecommand{\bysame}{\leavevmode\hbox to3em{\hrulefill}\thinspace}
\providecommand{\MR}{\relax\ifhmode\unskip\space\fi MR }
\providecommand{\MRhref}[2]{%
  \href{http://www.ams.org/mathscinet-getitem?mr=#1}{#2}
}
\providecommand{\href}[2]{#2}


\begin{dajauthors}
\begin{authorinfo}[Collin]
  Collin Bleak\\
  University of St Andrews\\
  Scotland, United Kingdom\\
  cb211\imageat{}st-andrews\imagedot{}ac\imagedot{}uk \\
  \url{https://orcid.org/0000-0001-5790-1940}
\end{authorinfo}
\begin{authorinfo}[Peter]
  Peter J. Cameron\\
  Professor\\
  University of St Andrews\\
  Scotland, United Kingdom\\
  pjc20\imageat{}st-andrews\imagedot{}ac\imagedot{}uk \\
  \url{https://orcid.org/0000-0003-3130-9505}
\end{authorinfo}
\begin{authorinfo}[Shayo]
  Feyishayo Olukoya\\
  University of St Andrews\\
  Scotland, United Kingdom\\
  fo55\imageat{}st-andrews\imagedot{}ac\imagedot{}uk \\
  \url{https://orcid.org/0000-0003-3285-9023}
\end{authorinfo}

\end{dajauthors}

\end{document}